\documentclass[12pt, reqno]{amsart}
\usepackage{fullpage,url,amssymb,amsmath,amsthm,amsfonts,mathrsfs}
\usepackage[usenames,dvipsnames]{color}
\usepackage[pagebackref = true, colorlinks = true, linkcolor = blue, citecolor = Green]{hyperref}
\usepackage[alphabetic,lite]{amsrefs}
\usepackage{enumitem}
\usepackage{amscd}   
\usepackage[all, cmtip]{xy} 

\DeclareFontEncoding{OT2}{}{} 

\newcommand{\edited}[1]{{#1}}
\newcommand{\Editnote}[1]{}

\newtheorem{lemma}{Lemma}[section]
\newtheorem{theorem}[lemma]{Theorem}
\newtheorem{prop}[lemma]{Proposition}
\newtheorem{cor}[lemma]{Corollary}

\newtheorem{thm}[lemma]{Theorem}
\newtheorem{defn}[lemma]{Definition}
\newtheorem{example}[lemma]{Example}

\newtheorem*{thmI}{Theorem I}

\newtheorem*{prop*}{Proposition}

\theoremstyle{definition}
\newtheorem{remark}[lemma]{Remark}
\newtheorem{remarks}[lemma]{Remarks}

\newcommand{\A}{{\mathbb A}}
\newcommand{\G}{{\mathbb G}}

\newcommand{\PP}{{\mathbb P}}

\newcommand{\F}{{\mathbb F}}
\newcommand{\Q}{{\mathbb Q}}

\newcommand{\Z}{{\mathbb Z}}

\newcommand{\Qbar}{{\overline{\Q}}}

\newcommand{\Xbar}{{\overline{X}}}
\newcommand{\Cbar}{{\overline{C}}}

\newcommand{\Ebar}{{\overline{E}}}

\newcommand{\kk}{{\mathbf k}}

\newcommand{\mm}{{\mathfrak m}}
\newcommand{\frakL}{\mathfrak L}
\newcommand{\frakS}{\mathfrak S}

\newcommand{\calC}{{\mathcal C}}

\newcommand{\calE}{{\mathcal E}}

\newcommand{\calP}{{\mathcal P}}

\newcommand{\OO}{{\mathcal O}}

\DeclareMathOperator{\rk}{rk}

\DeclareMathOperator{\inv}{inv}

\DeclareMathOperator{\im}{im}

\DeclareMathOperator{\Gal}{Gal}

\DeclareMathOperator{\Cor}{Cor}
\DeclareMathOperator{\Res}{Res}
\DeclareMathOperator{\Norm}{Norm}

\DeclareMathOperator{\Br}{Br}

\DeclareMathOperator{\divv}{div}

\DeclareMathOperator{\Div}{Div}

\DeclareMathOperator{\Pic}{Pic}
\DeclareMathOperator{\NS}{NS}
\DeclareMathOperator{\Jac}{Jac}
\DeclareMathOperator{\Num}{Num}

\DeclareMathOperator{\Spec}{Spec}

\DeclareMathOperator{\HH}{H}
\DeclareMathOperator{\hh}{h}

\DeclareMathOperator{\res}{res}
\DeclareMathOperator{\Princ}{Princ}
\DeclareMathOperator{\rank}{rank}

\newcommand{\isom}{\cong}

\newcommand{\To}{\longrightarrow}
\newcommand{\BOfl}{B^{0,\textup{fl}}}
\newcommand{\Bfl}{B^{\textup{fl}}}

\numberwithin{equation}{section}
\numberwithin{table}{section}

\newcommand{\defi}[1]{\textsf{#1}} 

\title{On Brauer groups of double covers of ruled surfaces}

\author{Brendan Creutz}
\address{Department of Mathematics and Statistics, University of Canterbury, Private Bag 4800, Christchurch 8140, New Zealand}
\email{brendan.creutz@canterbury.ac.nz}
\urladdr{http://www.math.canterbury.ac.nz/\~{}b.creutz}

\author{Bianca Viray}
\thanks{The second author was partially supported by NSF grant DMS-1002933.}
\address{University of Washington, Department of Mathematics, Box 354350, Seattle, WA 98195, USA}
\email{bviray@math.washington.edu}
\urladdr{http://math.washington.edu/\~{}bviray}

\date{}
\subjclass{14F22, 14G05}

\begin{document}
	\begin{abstract}
	Let $X$ be a smooth double cover of a geometrically ruled surface defined over a separably closed field of characteristic different from $2$. The main result of this paper is a finite presentation of the $2$-torsion in the Brauer group of $X$ with generators given by central simple algebras over the function field of $X$ and relations coming from the N\'eron-Severi group of $X$. In particular, the result gives a central simple algebra representative for the unique nontrivial Brauer class on any Enriques surface.  An example demonstrating the applications to the study of rational points is given.
	\end{abstract}
	\maketitle
	\section{Introduction}

	\subsection{Outline of the main results}
		Let $X_\circ$ be a smooth, projective and geometrically integral variety over a field $k_\circ$. Let $k$ be a separable closure of $k_\circ$ and let $X$ denote the base change of $X_\circ$ to $k$. The \defi{Brauer group} of $X_\circ$, denoted $\Br X_\circ$, is a generalization of the usual notion of the Brauer group of a field. Our results concern the $2$-torsion in $\Br X_\circ$ for $X_\circ$ a desingularization of a double cover of a ruled surface. Up to birational equivalence, this class of varieties contains all surfaces with an elliptic or hyperelliptic fibration, all double covers of $\PP^2$, and (at least over a separably closed field) all Enriques surfaces.  Under fairly mild assumptions, we obtain a finite presentation of the $\Gal(k/k_\circ)$-module $\Br X[2]$ in terms of generators given by unramified central simple algebras over the function field $\kk(X)$ and relations coming from the N\'eron-Severi group of $X$. We prove the following.
		
		 \begin{thmI}
			Let $S_\circ \to W_\circ$ be a geometrically ruled surface defined over a field $k_\circ$ of characteristic different from $2$, and let $X_\circ$ be \edited{a}\Editnote{Comment 15} desingularization of a double cover of $S_\circ$ branched over a reduced curve $B_\circ$ that is flat over $W_\circ$ and has at worst simple singularities. Let $\frak{L}_{c,\calE} \subseteq \kk(B)^\times/\kk(W)^\times\kk(B)^{\times 2}$ denote the \edited{finitely generated} subgroup defined in \S\ref{sec:BrX2presentation}. The map $\gamma':\kk(B) \to \Br \kk(X)$ defined in \S\ref{subsec:curves} induces an exact sequence of $\Gal(k/k_\circ)$-modules,
			\[
				\frac{\NS X}{2 \NS X}\stackrel{x - \alpha}{\To}
				\frakL_{c,\calE}\stackrel{\gamma}{\To}
				\Br X[2] \To 0\,,
			\]
			where $\NS X$ denotes the N\'eron-Severi group of $X$.
		\end{thmI}		
		\noindent(See also Theorems~\ref{thm:BrX2presentation} and~\ref{thm:BrauerEnriques} and Corollaries~\ref{cor:BrX2presentation} and~\ref{cor:DoubleCoverOfP2} for related results.) 

		This result enables a study of the Galois action on $\Br X[2]$, and, in many cases, computation of $\Br X_\circ[2]$. Consequently, we expect it to have important arithmetic applications. More precisely, if $k_\circ$ is a global field, then, as Manin~\cite{Manin-BMobs} observed, elements of the Brauer group can obstruct the existence of $k_\circ$-points, even when there is no local obstruction.  Computation of such an obstruction requires explicit representations of the elements of $\Br X_\circ$; knowledge of the group structure alone does not suffice.
		
		Our method allows us to write down an unramified central simple algebra representing the nontrivial Brauer class on any Enriques surface, which we then use to give a numerical example of an Enriques surface with a transcendental Brauer-Manin obstruction to weak approximation.  The example also demonstrates how our method can be used to obtain Brauer class representatives defined over a global field.
		
		As a further application we show that the presentation can be used to determine the size of $\Br X[2]$ without the aid of the exponential sequence or knowledge of the Betti numbers. For example, a double cover of a quadric surface branched along a $(4,4)$ curve is a K3 surface for which we recover the well known fact that $\Br X[2]$ has $\F_2$-dimension $22 - \textup{rank}\NS X$.
		
	\subsection{Discussion}	
		The key feature enabling our results is the fibration on $X_\circ$ induced by the ruled surface, a fibration whose generic fiber is a double cover $C_\circ \to \PP^1$. To understand the relevance of this, recall that the Brauer group of $X_\circ$ admits a filtration, $\Br_0 X_\circ \subseteq \Br_1 X_\circ \subseteq \Br X_\circ$, where $\Br_0X_\circ := \im\left(\Br k_\circ\to\Br X_\circ\right)$ is the subgroup of \defi{constant Brauer classes} and $\Br_1X_\circ := \ker\left(\Br X_\circ \to \Br X\right)$ is the subgroup of \defi{algebraic Brauer classes}. Using the Hochschild-Serre spectral sequence, the algebraic classes can be understood in terms of the Galois action on the Picard group of $X$. In contrast, computation of \defi{transcendental Brauer classes}, i.e. those surviving in the quotient $\Br X_\circ/\Br_1X_\circ$, is usually much more difficult, with only a handful of articles addressing the problem~\cites{Harari-transcendental, Wittenberg-transcendental, SSD-2descent, HS-Enriques, Ieronymou-transcendental, KT-effectivity, HVAV-K3, SkorobogatovZarhin, HVA-K3Hasse,Preu-transcendental}.

		In the context of a fibration as above, the purity theorem~\cite[Thm. 6.1]{Grothendieck-BrauerIII}\Editnote{Comment 9} gives a filtration
		\[
			\Br X \subseteq \Br C \subseteq \Br \kk(C) = \Br \kk(X)\,,
		\]
		identifying $\Br C \subseteq \Br \kk(X)$ as the subgroup unramified at all horizontal divisors, and $\Br X \subseteq \Br C$ as the subgroup unramified at all vertical divisors. Moreover, the Brauer group of $C$ is algebraic (over the function field of the base curve) by Tsen's theorem, and so may be studied in terms of the Galois action on the Picard group.  In~\cite{CV-BrauerCurves} this fact is utilized to obtain a presentation of $\Br C[2]$, which is induced by the map $\gamma':\kk(B) \to \Br\kk(X)$ in Theorem I. The task of the present paper is, thus, to determine which functions in $\kk(B)$ give rise to central simple algebras that are unramified at the vertical divisors.
		
		That a fibration can be used in this way to compute Brauer classes on a surface is not new, but there are few classes of surfaces for which the method has been carried out in practice. Our work builds on that of Wittenberg~\cite{Wittenberg-transcendental} and Ieronymou~\cite{Ieronymou-transcendental} who each give an example of a nontrivial transcendental $2$-torsion Brauer class on a specific elliptic K3 surface. The surfaces they consider admit a genus one fibration such that the Jacobian fibration has full rational $2$-torsion and such that the generic fiber is a double cover of $\PP^1$. We apply the results of~\cite{CV-BrauerCurves} to formalize and generalize the technique to deal with fibrations of curves of arbitrary genus that are double covers of $\PP^1$ and remove all assumptions on the Jacobian fibration.

	\subsection{Outline}
		The notation used throughout the paper is defined in \S\ref{sec:Notation}, where the relevant results from~\cite{CV-BrauerCurves} (including the definition of the maps $\gamma$ and $x-\alpha$ appearing in Theorem I) are also recalled. Following this in~\S\ref{sec:UnramConditions}, we review the purity theorem and set about computing the residues of classes in the image of $\gamma$ at the vertical divisors. In \S\ref{sec:BrX2presentation} we identify a list \edited{of}\Editnote{Comment 1} functions in $\kk(B)^\times$ whose images under $\gamma'$ are unramified, and then prove a more general version of Theorem I. In~\S\ref{sec:sizeBr}, we use the aforementioned list of functions to determine the $\F_2$-dimension of $\Br X[2]$ under some mild hypotheses. Next in~\S\ref{sec:Enriques}, we turn our attention to Enriques surfaces and show how to construct the nontrivial Brauer class on any Enriques surface.  This demonstrates that our method is effective even when $S$ fails to be geometrically ruled.  We demonstrate an arithmetic application in~\S\ref{sec:BMobs} by constructing an explicit Enriques surfaces with a transcendental Brauer-Manin obstruction to weak approximation
	
	\section{Notation and Background}\label{sec:Notation}
		
		\subsection{General Notation}
			Throughout the paper $k_\circ$ denotes a field of characteristic different from $2$ with separable closure $k$. If $Y_\circ$ is a $k_\circ$-scheme, we write $Y$ for the base change of $Y_\circ$ to $k$. In general, if $\F$ is any field and $Y$ and $S$ are $\F$-schemes, we set $Y_S := Y \times_{\Spec \F} S$. We also define $Y_A := Y_{\Spec A}$, for an $\F$-algebra $A$. If $Y$ is an integral $\F$-scheme, $\kk(Y)$ denotes its function field. More generally, if $Y$ is a finite union of integral $\F$-schemes $Y_i$, then $\kk(Y) := \prod \kk(Y_i)$ is the ring of global sections of the sheaf of total quotient rings. In particular, if $A \simeq \prod \F_j$ is an \'etale $\F$-algebra, then $Y_A$ is a union of integral $\F$-schemes and $\kk(Y_A) \simeq \prod\kk(Y_{\F_j})$. For $r \ge 0$ we use $Y^{(r)}$ to denote the set of codimension $r$ points on $Y$. 			
			
			For any projective variety $Y$ over $\F$, we write $\Br Y$ for the \'etale cohomology group $\Br Y := \HH^2_\textup{\'et}(Y,\G_m)$. If $A$ is an \'etale $\F$-algebra, then we also write $\Br A$ for $\Br (\Spec A)$. Given invertible elements $a$ and $b$ in an \'etale $\F$-algebra $A$, \edited{and assuming the characteristic of $\F$ is not $2$,}\Editnote{Comment 2
} we define the quaternion algebra, 
			\[
				(a,b)_2 := A[i,j]/\langle i^2 = a,\, j^2 = b,\, ij = -ji\rangle\,
			\]
			which we often conflate with its class in $\Br A$.
			
			\edited{We let $\Div Y$ denote the group of Weil divisors of $Y$ defined over $\edited{\F}$. If $Y$ is smooth and geometrically irreducible, we} let $\Pic Y$ be the Picard group of $Y$ and $\Pic_Y$ for its Picard scheme. Then $\Pic Y = \Div Y/\Princ Y$, where  $\Princ Y$ is the group of  principal divisors of $Y$ defined over $\edited{\F}$\Editnote{Comment 1}. 
            If $D \in \Div Y$, then $[D]$ denotes its class in $\Pic Y$. If $\Br \F = 0$, then in addition $\Pic_Y (\F) = \Pic Y$. However, for general fields, the map $\Pic Y \to \Pic_Y(\F)$ is not necessarily surjective. Let $\Pic^0_Y \subseteq \Pic_Y$ denote the connected component of the identity, and use $\Pic^0Y$ to denote the subgroup of $\Pic Y$ mapping into $\Pic^0_Y(\F)$. Then $\NS Y := \Pic Y/\Pic^0 Y$ is the {N\'eron-Severi group} of $X$. If $\lambda \in \NS Y_{\F^\textup{sep}}$, let $\Pic^\lambda_Y$ denote the corresponding component of the Picard scheme and use $\Pic^\lambda Y$ and $\Div^\lambda Y$ to denote the subsets of $\Pic Y$ and $\Div Y$ mapping into $\Pic^\lambda_Y(\F)$. When $Y$ is a curve $\NS Y_{\F^\textup{sep}} = \Z$ and $\Jac(Y) := \Pic^0_Y$ is called the Jacobian of $Y$. It is an abelian variety of dimension $g(Y)$, where $g(Y)$ denotes the genus of $Y$. If $Y$ is a \edited{disjoint} finite\Editnote{Comment 3} union of integral curves $Y_i$ we define 
            $\Jac(Y) := \prod \Jac(Y_i)$ and $g(Y) := \sum g(Y_i) + 1 - h^0(Y)$, where $h^0(Y)$ is the number of connected components of $Y$.

		\subsection{Double covers of ruled surfaces}
		\label{subsec:SurfacesNotation}
			
			Let $\varpi\colon S\to W$ be a \edited{smooth projective }ruled surface over $k$, i.e., a \edited{proper} fibration over a smooth irreducible \edited{projective} curve $W$ which is defined over $k$ and such that the generic fiber of $\varpi$ is isomorphic to $\PP^1_K$, where $K := \kk(W)$.\Editnote{Comment 4} \edited{We fix an isomorphism $S_K\stackrel{\sim}{\to} \PP^1_K$ and } write $\mathfrak{S}$ for the flat closure of $\infty \in \PP^1_K$ in $S$.  We say that $S$ is \defi{geometrically ruled} if \emph{every} fiber of $\varpi$ is isomorphic to $\PP^1$. 
			
			Let $\pi\colon X^0\to S$ be a double cover defined over $k$, and let $X$ be the desingularization of $X^0$ that is obtained by a \defi{canonical resolution} (see~\cite[\S III.7]{BHPvdV} for the definition). The composition $X\to S\to W$ endows $X$ with a fibration whose generic fiber $C$ is a double cover of $\PP^1_K$. 			
			
			Let $B^0 \edited{\subset S}$ denote the union of the {connected} components of the branch locus of $X^0\to S$ that map dominantly to $W$. We will assume that $B^0$ is nonempty, which is equivalent to assuming \edited{that the ramification locus of $C\to\PP^1$ is nonempty, which in particular implies that $X$ is geometrically irreducible.  (If the ramification locus of $C\to \PP^1$ is empty, then any irreducible component of $X$ is a ruled surface and so has trivial Brauer group.)}\Editnote{Comment 6}           
				
			The restriction of the map $\varpi\colon S \to W$ to ${B^0}$ may not be flat as ${B^0}$ may have vertical \edited{irreducible}\Editnote{Comment 7} components.  Let $B^{0,\textup{fl}}\subseteq {B^0}$ be the maximal subvariety such that the map $B^{0,\textup{fl}}\to W$ is flat, i.e., $B^{0,\textup{fl}}$ is the union of all \emph{irreducible} components of \edited{$B^0$}\Editnote{Comment 8} that map dominantly to $W$. We write ${B}$ for the normalization of ${B^0}$ in $X$, and write $B^{\textup{fl}}$ for the normalization of $B^{0,\textup{fl}}$ in $X$.  We set $L := \kk(B^\textup{fl})$ and note that $\kk(B) = L\times k(x)^n$, where $n$ equals the number of vertical \edited{irreducible} components of $B^0$.  We denote the normalization map $B \to B^0$ by $\nu$ and sometimes conflate $\nu$ with $\nu|_{\Bfl}$.
			
			Since $X$ was obtained by a canonical resolution, the curves $B$ and $B^{\textup{fl}}$ are smooth~\cite[\S III.7]{BHPvdV}. In particular, $B$ and $B^{\textup{fl}}$ are each a disjoint union of integral curves. 
				
			If $b\in \BOfl$ is a point lying over $w\in W$ and $\ell\in L^{\times}$, we define
			\[
				v_b(\ell) := \sum_{b'\in \Bfl, b'\mapsto b} v_{b'}(\ell), \quad
				e(b/w) := \sum_{b'\in \Bfl, b'\mapsto b} e(b'/w),
			\]
			where $e(b'/w)$ denotes the ramification index of the map $\Bfl \to W$ at $b'$ and $v_{b'}(\ell)$ is the valuation of $\ell$ at $b'.$

            \edited{
            We will consider \defi{the dual graph} $\Gamma$ of $B^0$ defined as follows. For every singular point $b\in B^0$, fix an ordering of the preimages $b_0', \ldots, b_s'\in B$ of $b$.  We define the vertices of $\Gamma$ to be in one-to-one correspondence with the irreducible components of $B$, and define the edges of $\Gamma$ by the following rule: for every singular point $b \in B^0$ and every $1\leq i\leq s$, there is an edge $e_{b,i}$ joining the vertices corresponding to the irreducible components containing $b_{i-1}'$ and $b_i'$. 

        	\begin{remark}
        		Strictly speaking it is not correct to refer to {\em the} dual graph of $B$, since $\Gamma$ depends on the ordering chosen above. However, its homology with coefficients in $\Z/2\Z$ does not; as this is all we are really concerned with below, we will allow ourselves this abuse of language.
        	\end{remark}

        	We will work with the simplicial homology with coefficients in $\Z/2\Z$ of the dual graph $\Gamma$. We use $C_i(\Gamma)$, $Z_i(\Gamma)$ and $H_i(\Gamma)$ to denote, respectively, the groups of $i$-chains, $i$-cycles and $i$-homology classes for $i = 0,1$.
            }\Editnote{This passage previously appeared (with minor alterations) in Section 4.1.}

			Since the branch locus of $\pi$ is generically smooth and $S$ is smooth, $X^0$ is regular in codimension $1$.  Let $\calE$ be the set of curves on $X$ that are either contracted to a point in $X^0$, or lie over some $w\in W$ such that $S_w$ is singular. Since $X^0$ is regular in codimension $1$, the morphism $X\to X^0$ is an isomorphism away from $\calE$.  We say that an irreducible curve $F$ on $X$ is \defi{exceptional} if $F\in\calE$ and \defi{non-exceptional} otherwise.  If $F$ is a curve on $X^0$, we will often abuse notation and let $\partial_F$ denote the residue map at the strict transform of $F$ on $X$.
			\begin{remark}\label{rmk:exceptional}
				There are some curves in $\calE$ which are not ``exceptional'' in the usual sense, i.e., are not the exceptional divisor of some blow-up.  Such curves are all components of $X^0_w$ for some $w\in W$ such that $S_w$ is not smooth.  In particular, if $S$ is geometrically ruled, 
                then every curve in $\calE$ is the exceptional divisor of some blow-up.
			\end{remark}
		
			\begin{remark}
				Many of the arguments are simpler and more intuitive when $B^0 = \BOfl$, and even more so under the additional assumption that ${B^0}$ is smooth and irreducible. As many of the results are of equal interest in these cases, the reader may wish to make these assumptions on a first reading.
			\end{remark}
			
			\begin{remark}\label{rem:mildersing}
				Any double cover of a ruled surface is in fact birational to a double cover of a \emph{geometrically} ruled surface of the form $\PP^1\times W$. However, allowing for a more general ruled surface enables us to choose a model where the branch locus has milder singularities. This will be used when we consider the Brauer group of an Enriques surface in \S\ref{sec:Enriques}.
			\end{remark}

		\subsection{The $2$-torsion Brauer classes on a double cover of the 
        projective line}\label{subsec:curves}
			Recall that the generic fiber of the composition $X \to S \to W$ is a double cover $C \to \PP^1_K$.  As mentioned in the introduction, we rely on the explicit presentation for $\Br C[2]$ given in~\cite{CV-BrauerCurves}.  For the reader's convenience, the relevant results and notation are recalled here.
			
			\edited{Let $\Omega \subset C$ denote the ramification locus of $C \to \PP^1_K$. We shall assume that $\Omega$ does not contain any point above $\infty\in \PP^1_K$. Since $K$ is infinite, this can always be arranged by changing coordinates on $\PP^1_K$.}\Editnote{Comment 5}  Thus, we may fix a model for $C$ of the form $y^2 = cf(x)$, where $c\in K^\times$ and $f$ is squarefree and monic of degree $2g(C)+2$.
		
			Note that $L := \kk(\Bfl)$ can be identified with $K[\theta]/f(\theta)$. Let $x - \alpha$ denote the image of $x - \theta$ in $\kk(C_L) := L \otimes_K \kk(C)$. As shown in~\cite{PS-descent} the map sending a closed point $P\in C\setminus \left(\Omega\cup \pi^{-1}(\infty)\right)$ with $P(\overline{K}) = \{(x_1, y_1), \ldots, (x_d,y_d)\}$ to $x(P) - \alpha := \prod_{i=1}^{d} (x_i - \alpha)\in L^{\times}$ induces a homomorphism 
			\[
				x-\alpha\,:\,\Pic C \to L^\times/K^\times L^{\times 2}\,.
			\]
						
			Set $\frakL = L^{\times}/K^{\times}L^{\times2}.$  For $a \in K^{\times}$ and $\ell \in L^{\times}$, we denote their corresponding classes in $K^{\times}/K^{\times2}$ and $\frakL$ by $\overline{a}$, $\overline{\ell}$ respectively.  We define
			\[
				\frakL_a := \left\{\overline{\ell} \in \frakL : \Norm_{L/K}(\overline{\ell}) \edited{\in} \langle \overline{a}\rangle \right\}.
			\]
	
			Consider the map
			\[
				\gamma' \colon L^{\times} \to \Br \kk(C), \quad
				\ell\mapsto \Cor_{\kk(C_L)/\kk(C)}\left( (\ell, x - \alpha)_2 \right).
			\]
			
			\begin{theorem}[{\cite[Theorems 1.1-4 and Proposition 4.7]{CV-BrauerCurves}}]
				\label{thm:gammawelldefined}
				\label{thm:EvenHypThm}
				Let $\ell \in L^\times$. Suppose $k$ is a separable closure of a field $k_\circ$ over which $\pi,\varpi,c$ and $f(x)$ are all defined. Then $\gamma'(\ell) \in \Br C$ if and only if $\overline{\ell} \in \frakL_c$. Moreover, $\gamma'$ induces an exact sequence of $\Gal(k/k_\circ)$-modules,
				\[
					\frac{\Pic C}{2\Pic C} \stackrel{x-\alpha}\To \frak{L}_c \stackrel{\gamma}\To \Br C[2] \To 0\,,
				\]
				where the kernel of $x-\alpha$ is generated by divisors lying over $\infty \in \PP^1_K$.
			\end{theorem}	
			
	\section{Residues and Purity}\label{sec:UnramConditions}

		By the purity theorem we have that for any smooth projective variety $Y$ over a field of characteristic different from a prime $p$, 
		\begin{equation}
			\label{eq:purity}
			\Br Y[p] = \bigcap_{y\in Y^{(1)}} \ker\left(\Br \kk(Y)[p] \stackrel{\partial_y}{\To} \HH^1(\kk(y), \Z/p\Z)\right),
		\end{equation}
		where $\partial_y$ denotes the residue map associated to $y \in Y^{(1)}$ (see~\cite[Thm. 6.1]{Grothendieck-BrauerIII}). For quaternion algebras, the residue map $\partial_y$ can be described explicitly (see~\cite{GS-csa}*{Example 7.1.5}).  For any $a,b\in \kk(Y)^{\times}$, we have
		\[
			\partial_y\left((a,b)_2\right) = (-1)^{v_y(a)v_y(b)}a^{v_y(b)}b^{-v_y(a)} \in \kk(y)^{\times}/\kk(y)^{\times2} \isom \HH^1\left(\kk(y), \Z/2\right),
		\]
		where $v_y$ denotes the valuation corresponding to $y$.
		
		Identifying $\kk(C) \simeq \kk(X)$ we have
		\begin{align*}
			\Br C [2] &= \bigcap_{\substack{\sigma \in X^{(1)}\\ \sigma \textup{ horizontal}}}\ker\left(\Br\kk(X)[2] \stackrel{\partial_\sigma}\To \kk(\sigma)^\times/\kk(\sigma)^{\times 2}\right) \,,\\
		\intertext{and}
			\Br X [2] &= \bigcap_{\substack{F \in X^{(1)}\\ F \textup{ vertical}}}\ker\left(\Br C [2] \stackrel{\partial_F}\To \kk(F)^\times/\kk(F)^{\times 2}\right) \,.
		\end{align*}
		In light of this we may obtain a presentation of $\Br X[2]$ from Theorem~\ref{thm:EvenHypThm} as soon as we can determine the subgroups $\ker(\partial_F\circ \gamma)$ for $F \in X^{(1)}$ a vertical divisor.  To do so, we will often use that
			\begin{equation}\label{eq:purity-cores}
				\partial_F\left( \Cor_{\kk(C_L)/\kk(C)}\left((\ell, x - \alpha)_2\right)\right) = \prod_{F'|F} \Norm_{\kk(F')/\kk(F)} \left(\partial_{F'} \left((\ell, x - \alpha)_2\right)\right)
			\end{equation}
			for all $F\in X^{(1)}$ and all $\ell\in \edited{L^\times}$\Editnote{Comment 1}; here $F'$ runs over all discrete valuations of $\kk(C_L)$ that extend the discrete valuation corresponding to $F$ on $\kk(C) = \kk(X)$ (see~\cite[Lemma 2.1]{CV-BrauerCurves}).

		\subsection{The non-exceptional \edited{vertical} curves}
			\begin{prop}\label{prop:UnramifiedConditions}
				Fix $w\in W$ such that $S_w$ is smooth and fix $\ell\in L^{\times}$.\Editnote{The assumption that $\overline{\ell}\in\frakL_c$ has been removed as it was not needed in the proof.}
				\begin{enumerate}
					\item If $X^0_{w}$ is reduced and irreducible, then $\partial_{{X^0_{w}}}(\gamma(\overline{\ell})) \in \kk({X^0_{w}})^{\times2}$ if and only if 
					\begin{equation}\label{eq:unram-irred}
						e(b'/w)v_b(\ell)\equiv e(b/w)v_{b'}(\ell)\bmod 2\,, \quad
						\textup{ for all }b,b'\in B^0_{w} \setminus (B^0_{w} \cap \mathfrak{S}).
					\end{equation}
					\item If $X^0_{w}$ is reduced and reducible\Editnote{Comment 12}, then $\partial_{F}(\gamma(\overline{\ell})) \in \kk(F)^{\times2}$ for all irreducible components $F\subseteq X^0_{w}$ if and only if
					\begin{equation}\label{eq:unram-red}
						v_b(\ell)\equiv 0 \bmod 2\,, \quad
						\textup{ for all }b\in B^0_{w}\setminus (B^0_{w} \cap \mathfrak{S}).
					\end{equation}
					\item If $S_w \subseteq B^0$, then $\partial_{(X^0_w)_{\textup{red}}}(\gamma(\overline{\ell})) \in \kk((X^0_w)_{\textup{red}})^{\times2}$ for all $\ell\in L^{\times}$.
				\end{enumerate}
			\end{prop}
			
			\begin{cor}
				\label{cor:UnramifiedOutsideE}
				Let $\overline{\ell} \in \frakL_c$. Then $\gamma(\overline{\ell}) \in \Br(X \setminus \calE)$ if and only if some (equivalently every) representative of $\overline\ell$ satisfies~\eqref{eq:unram-irred} at every $w \in W$ such that $X^0_w$ is reduced and irreducible and $S_w$ is smooth and satisfies~\eqref{eq:unram-red} at every $w \in W$ such that $X^0_w$ is reduced and reducible and $S_w$ is smooth.
			\end{cor}
			\begin{proof} 
				Every non-exceptional vertical curve maps dominantly to a smooth and irreducible $S_w$ for some $w \in W$.  If $S_w$ is smooth and irreducible, then $X^0_w$ is \edited{non-}reduced\Editnote{Comment 13} if and only if $S_w\subseteq B^0$.  Therefore, for every $F\in X^{(1)}\setminus \calE$, Proposition~\ref{prop:UnramifiedConditions} gives necessary and sufficient conditions for $\partial_F(\ell)\in \kk(F)^{\times2}.$  This is exactly the content of the Corollary.
			\end{proof}
			
			\begin{proof}[Proof of Proposition~\ref{prop:UnramifiedConditions}]
				Fix $\ell\in L^{\times}$,\edited{ $w\in W$ such that $S_w$ is smooth,} and let $F\subseteq X^0_{w}$ be a reduced and irreducible curve.  By~\eqref{eq:purity-cores}, we have 
				\begin{equation}
					\label{eq:resatF}
					\partial_{F}(\gamma(\overline{\ell})) = 
					\prod_{\substack{F'\in (X_{\Bfl}')^{(1)}\\ 
					F' \mapsto F \textup{ dominantly}}}
					\Norm_{\kk(F')/\kk(F)}
					(\ell^{w'(x-\alpha)}(x-\alpha)^{\edited{-}w'(\ell)}).
				\end{equation}\Editnote{Comment 14}
				Here $X_{\Bfl}'$ denotes \edited{a}\Editnote{Comment 15} desingularization of $X_{\Bfl} := X\times_{W}\Bfl$, and $w'$ denotes the valuation associated to $F'$.  The surface $X\times_{W}\Bfl$ is regular at all codimension $1$ points lying over $w\in W$ such that $X^0_w$ is reduced.
	
				Assume that $X^0_w$ is not reduced, or, equivalently, that $S_w\subseteq B^0$.  Then the map on residues $\HH^1(\kk(S_w), \Q/\Z) \to \HH^1(\kk(X^0_w), \Q/\Z)$ is identically zero on $2$-torsion classes.  Since $\gamma'(\ell)\in \im\left(\res \Br \kk(S)\to \Br \kk(X)\right)$, the residue $\partial_{(X^0_w)_{\textup{red}}}(\edited{\gamma'(\ell)}) \in \kk((X^0_w)_{\textup{red}})^{\times2}$\Editnote{Comment 16} for all $\ell\in L^{\times}$.
					
				Henceforth, we assume that $X^0_w$ is reduced, or, equivalently, that $\BOfl_w = B^0_w$.   Then, since $X\times_{W}\Bfl$ is regular at all codimension $1$ points above $w$, the prime divisors of $X_{\Bfl}'$ that map dominantly to $F$ are \edited{in one-to-one correspondence with} the prime divisors of $X_{\Bfl}$ that map dominantly to $F.$ 

				To compute the residues at $X^0_w$, we will need to have a model of the fiber.  For this, we will use the following lemma.
				\begin{lemma}
					For every $w\in W$ such that $S_w$ is smooth, there exists an open set $U\subseteq W$ containing $w$ and constants $a \edited{\in K^{\times}},b\in K$ such that
					\[
						S_U \stackrel{\sim}{\to}\PP^1_k\times U, s\mapsto (ax(s) + b, \varpi(s)).
					\]
				\end{lemma}
				\begin{proof}
					By the Noether-Enriques theorem~\cite[Thm. III.4]{Beauville-CAS}, there is an isomorphism $\varphi\colon S_U\to \PP^1\times U$ which commutes with the obvious morphisms to $U$.  After possibly \edited{shrinking $U$ and possibly} composing with an automorphism of $\PP^1_U$, we may assume that $\mathfrak{S}$ maps to $\{\infty\}\times U$\Editnote{Comment 11}.  To complete the proof we observe that $\varphi$ must induce an automorphism of $\PP^1_K$ that preserves $\infty$.
				\end{proof}
				Fix $U\subseteq W$, $a \edited{\in K^{\times}},b\in K$ as in the lemma.   Note that, the algebra $\Cor_{\kk(C_L)/\kk(C)}\left((\ell, a)_2\right)$ is constant, and hence trivial.  Therefore, $\Cor_{\kk(C_L)/\kk(C)}\left((\ell, ax + b - (a\alpha + b))_2\right) = \gamma'(\ell)$. 
				
				{Thus, by replacing $x$ with $ax + b$, $\alpha$ with $a\alpha + b$ and $f$ with $f(x/a - b/a)$ if necessary, we may assume that $x$ is a horizontal function}, i.e. that it has no zeros or poles along any fibers of $U$, and that it restricts to a non-constant function along any fiber of $U$.
				Then the function $x - \alpha$ has non-positive valuation along any fiber of $X_{\Bfl_U}$, and it has negative valuation on the fibers of $X_{\Bfl_U}$ where $\alpha$ has negative valuation.
								
                \edited{
                Let $F'$ be a prime divisor of $X_{\Bfl}$ that maps dominantly to $F$, let $w'$ be the associated valuation, and let $b\in \Bfl_w$ be the point that $F'$ lies over; note that $\kk(F') = \kk(F_{\kk(b)})$.\Editnote{Comment 10}  If $w'(x - \alpha)$ is negative, then $b$ lies over $B^0_w\cap \frakS$ and the function $(x - \alpha)^{w'(\ell)}/\ell^{w'(x - \alpha)}$ reduces to a constant in $\kk(F')$ (since $w'(x) = 0$).  Therefore,
                \[
                \partial_{F}(\gamma(\overline{\ell})) = \prod_{
                b\in B^0_{w}\setminus (B^0_{w} \cap \mathfrak{S})}
                \Norm_{\kk(b)/\kk(w)}(x - \alpha(b))^{v_b(\ell)}.
                \]
                Since $\kk(F) = k(x)\left(\sqrt{\prod_{b\in B^0_{w}\setminus (B^0_{w} \cap \mathfrak{S})}\Norm_{\kk(b)/\kk(w)}(x - \alpha(b))^{e(b/w)}}\right)$, $\gamma(\overline{\ell})$ is unramified at $F$ if and only if there exists some integer $m$ such that
                \begin{equation}\label{eq:UnramifiedReduced}
                    v_b(\ell) + me(b/w) \equiv 0 \bmod 2, \quad\textup{for all }
                    b\in B^0_{w}\setminus (B^0_{w} \cap \mathfrak{S}).
                \end{equation}
                
                Assume that $X^0_{w}$ is reducible. Then $e(b/w) \equiv 0 \bmod 2$ for all $b\in B^0_{w}\setminus (B^0_{w} \cap \mathfrak{S})$.  Hence~\eqref{eq:UnramifiedReduced} is equivalent to~\eqref{eq:unram-red}.  If $X^0_{w}$ is irreducible, then there exists a $b_0\in B^0_{w}\setminus (B^0_{w} \cap \mathfrak{S})$ such that $e(b_0/w)\equiv 1\bmod 2$, so~\eqref{eq:UnramifiedReduced} is equivalent to the matrix 
		
                \[
                    \begin{pmatrix} 
                        v_b(\ell) & e(b/w)
                    \end{pmatrix}_{b\in B^0_{w}\setminus 
                    (B^0_{w} \cap \mathfrak{S})}
                \]    
                having rank $1$.  This is clearly equivalent to~\eqref{eq:unram-irred}, which completes the proof.    
                }\Editnote{This rewrite addresses Comments 10, 17, 18, 19, and 20.}
			\end{proof}

		\subsection{Exceptional curves lying over simple singularities}
			\begin{prop}\label{prop:-2curves}
				Let $F$ be a $(-2)$-curve lying over a simple singularity of $X^0$.  If $A \in\Br \kk(X)[2]$ is unramified at all curves $F'\subseteq X$ that intersect $F$ and that are not contracted in $X^0$, then $A$ is unramified at $F$. In particular, if $S$ is geometrically ruled and $B^0$ has at worst simple singularities, then $\Br X[2] = \Br (X\setminus\calE)[2].$
			\end{prop}
			
			\begin{proof}[Proof of Proposition~\ref{prop:-2curves}]
				The canonical resolution of a simple singularity consists of a \edited{series of blow-ups}\Editnote{Comment 1}.  Therefore, the preimage of a simple singularity $P\in X^0$ is a tree of $(-2)$-curves.  Consider any $(-2)$-curve $F$ which is a leaf of the tree (i.e., has valence $1$) and let $Q\in F$ be the (unique) point which intersects another curve in the tree.
				
				As a special case of the Bloch-Ogus arithmetic complex~\cite[\S1, Prop. 1.7]{Kato}, we have the complex
				\[
					\Br \kk(X)[2] \stackrel{\oplus\partial_{F'}}\To \bigoplus_{F'\in X^{(1)}}\frac{\kk(F')^{\times}}{\kk(F')^{\times 2}} \To \bigoplus_{P\in X^{(2)}} \Z/2\Z.
				\]
				Therefore, for every codimension $2$ point $P\in X$, we have
				\[
					\sum_{\substack{F'\in X^{(1)} \\\text{with } P\in\overline{F'}}} v_P(\partial_{F'}(A)) \equiv 0 \bmod 2.
				\]
				{By assumption $\partial_{F'}(A) \in \kk(F')^{\times 2}$} for all $F'\in X^{(1)}$ whose closure intersects $F$ and is not contracted in $X^0$. Hence $v_P(\partial_F(A)) \equiv 0 \bmod 2$ for all $P\in X^{(2)}$ such that $P\in F$ and $P\neq Q$.  Since {$F$ is rational, this implies $\partial_F(A) \in \kk(F)^{\times 2}$.}  Therefore $A$ is unramified at all $(-2)$-curves which are leaves of the tree.  
				
				The same proof then shows that $A$ is unramified at all $(-2)$-curves $F$ such that all the children of $F$ are leaves.  Then we apply the same argument to all curves $F$ such that all of the grandchildren of $F$ are leaves, and so on, until we have shown that $A$ is unramified at all curves in the tree.
				
				For the final claim, we note that if $S$ is geometrically ruled and $B^0$ has at worst simple singularities, then $\calE$ consists of $(-2)$-curves lying over simple singularities of $X^0$.
			\end{proof}

	\edited{\section{Candidate functions}\label{sec:Functions}}
    \Editnote{This section has been added to address Comment 21-28}
        Consider the composition
		\[
			\psi: \kk(B)^\times \To \frac{\kk(B)^\times}{\kk(B)^{\times 2}} \stackrel{\divv}{\To} \Div(B) \otimes \Z/2\Z \stackrel{\nu_*}\To \Div(B^0) \otimes \Z/2\Z\,.
		\]
		Fix a point $w_1\in W$ such that $X^{0}_{w_1}$ is reduced and irreducible, and define the group 
		\[
			\kk(B)_\calE := \left\{ \tilde\ell \in \kk(B)^\times : \psi(\ell) \in \langle \frakS\cap B^0, B^0_{w_1}\rangle \subset \Div(B^0)\otimes\Z/2\right\}\,.
		\]
        Let $L_{\calE}$ denote the image of $\kk(B)_{\calE}$ under the projection map $\kk(B)^{\times}\to \kk(\Bfl)^\times = L^{\times}$.

	In Section~\ref{sec:BrX2presentation} we will prove that the elements of $L_{\calE}$ satisfy
	condition~\eqref{eq:unram-irred} (resp.~\eqref{eq:unram-red}) for every $w\in W$ such that $S_w$ is smooth and $X^0_w$ is reduced and irreducible (resp. reducible).  Therefore the classes in $\gamma'(L_\calE)$ are unramified at all non-exceptional vertical curves.  
        
        The goal of this section is to characterize the elements of $L_{\calE}$.  Specifically, we will prove:
		\begin{prop*}
			The group $L_\calE$ admits a filtration
			\[
                0 \subset L^{\times 2}= G_4 \subset G_3 \subset G_2 \subset G_1 \subset G_0 = L_\calE\,,
			\]
			where 
			$G_3/G_4 \cong \Jac(B)[2], \;
                	G_2/G_3 \cong H_1(\Gamma), \;
                	G_1/G_2\isom(\Z/2)^{m_1},\;\textup{and } 
			G_0/G_1\isom(\Z/2)^{m_c}\;\\ \textup{with } m_1,m_c \in \{0,1\}.$
		\end{prop*}
		The isomorphisms in the proposition will be made explicit, providing generators for each subgroup $G_i$, as well as formulas for $m_1$ and $m_c$ (see Proposition~\ref{prop:Filtration}).
        
        \subsection{Principal divisors on $B$}

            Recall that $\Gamma$ denotes the dual graph of $B^0$, defined in Section~\ref{subsec:SurfacesNotation}.  Every edge of $\Gamma$ corresponds to a pair of points on $B$; this gives a homomorphism $\divv\colon C_1(\Gamma) \to \Div(B)\otimes \Z/2\Z$. As $C_0(\Gamma)$ is the free $\Z/2\Z$-module on the irreducible components of $B$ we may define a homomorphism $\deg:\Div(B) \to C_0(\Gamma)$ by taking the degree on each irreducible component. The group $H_0(\Gamma)$ is (isomorphic to) the free $\Z/2\Z$-module on the connected components of $B^0$ and we may define a homomorphism $\deg\colon\Div(B^0) \to H_0(\Gamma)$ by taking the degree on each connected component. We define $\calP$ and $\calP^0$ to be the kernels of the degree maps on $\Div(B)\otimes \Z/2\Z$ and $\Div(B^0)\otimes \Z/2\Z$, respectively. Putting this together gives a commutative and exact diagram. (One may easily check that the right two columns are exact and commutative; the rest follows from the snake lemma and the fact that $H_1(\Gamma) = Z_1(\Gamma)$.) 
        	\begin{equation}
        		\label{eq:homologysquare}
        		\xymatrix{
        			& 0\ar[d]&0\ar[d] && 0\ar[d] & \\
        			0\ar[r]
        			&H_1(\Gamma) \ar[r]\ar[d]
        			&C_1(\Gamma) \ar[rr]\ar[d]^{\divv}
        			&&\frac{C_1(\Gamma)}{Z_1(\Gamma)} \ar[r]\ar[d]^{\delta}&0\\
        			0\ar[r]
        			&\calP \ar[r]\ar[d]
        			&\Div(B)\otimes \Z/2\Z \ar[rr]^\deg\ar[d]^{\nu_*}
        			&&C_0(\Gamma) \ar[r]\ar[d]&0\\
        			0\ar[r]
        			&\calP^0 \ar[r]\ar[d]
        			&\Div(B^0)\otimes\Z/2\Z \ar[rr]^\deg\ar[d]
        			&&H_0(\Gamma) \ar[r]\ar[d]&0\\
        			&0&0&&0			 
        		}
        	\end{equation}

        \begin{lemma}\label{lem:PrincipalDivisors}
        	Suppose $D \in \Div(B)$. The following statements are equivalent.
        	\begin{enumerate}
        		\item $\nu_*(D)$ has even degree on every connected component of $B^0$.
        		\item There exists $D' \in \Div(B)$ with even degree on every irreducible component of $B$ and such that $\nu_*(D-D') \in 2\Div(B^0)$.
        		\item There exists a principal divisor $D' \in \Div(B)$ such that $\nu_*(D-D') \in 2\Div(B^0)$.
        	\end{enumerate}
        \end{lemma}

        \begin{proof}
        	$(1) \Rightarrow (2)$: This follows from a diagram chase. If $D$ is in the middle and maps to $0$ in the lower right, then we can modify $D$ by $\divv(\gamma)$ for some $\gamma \in C_1(\Gamma)$ to get a divisor with even degree on every irreducible component.\\
        	$(2) \Rightarrow (3)$: This follows from the fact that the Jacobian of $B$ is a $2$-divisible group.\\
        	$(3) \Rightarrow (1)$: This follows from the fact that a principal divisor has degree $0$ on every irreducible component.
        \end{proof}

        \begin{prop}\label{prop:Existence}
			\hfill
			\begin{enumerate}
				\item For each $D\in\Div(B)$ such that $[D] \in \Jac(B)[2]$ there exists a function $\tilde\ell_{D} \in \kk(B)^\times$ such that $\divv(\tilde\ell_{D}) = 2D$.
				\item For each cycle $\calC \in H_1(\Gamma)$ there exists a function $\tilde\ell_\calC \in \kk(B)^\times$ such that 							\[\divv(\tilde\ell_\calC) \equiv \divv(\calC) \pmod {2\Div(B)}\,.\]
				\item For every $(n_c, n_{1})\in \Z/2\Z^2$, there exists a function $\tilde\ell\in\kk(B)^{\times}$ such that 
                \[
                    \nu_*\divv(\ell) \equiv 
                    n_c(\frak{S} \cap B^0) + n_{1}B^0_{w_1} \pmod{2\Div(B^0)}
                \]
                if and only if $n_c\frakS + n_1S_{w_1}$ intersects every connected component of $B^0$ with even degree.  If the functions corresponding to $(1,0),(0,1),$ and $(1,1)$ exist, then we denote them by $\tilde\ell_c$, $\tilde\ell_1$, and $\tilde\ell_{c,1}$ respectively.
			\end{enumerate}
        \end{prop}
        \begin{proof}
            The first statement follows from the definition of $\Jac(B)[2]$, and the second and third follow from Lemma~\ref{lem:PrincipalDivisors}.
        \end{proof}

        \begin{defn}
            For all $D \in \Div(B)$ such that $[D]\in\Jac(B)[2]$ and all $\calC\in \HH_1(\Gamma)$, we let $\ell_D$ and $\ell_{\calC}$ denote the images of $\tilde\ell_D$ and $\tilde\ell_\calC$, respectively, under the map $\operatorname{pr}:\kk(B)^{\times}\to L^{\times}$.  We define $\ell_c, \ell_1,$ and $\ell_{c,1}$ similarly, if $\tilde\ell_c$, $\tilde\ell_1$, and $\tilde\ell_{c,1}$ exist.
        \end{defn}

        \subsection{The filtration on $L_{\calE}$}

            In this section, we will prove a strengthened version of the proposition in the introduction.
    		\begin{prop}\label{prop:Filtration}
    			The group $L_\calE$ admits a filtration
    			\[
                    0 \subset G_4 \subset G_3 \subset G_2 \subset G_1 \subset G_0 = L_\calE\,,
    			\]
    			where 
    			\begin{enumerate}
     				\item $G_4 = L^{\times2}$,
    				\item the map $D\mapsto \ell_D$ induces an isomorphism $\Jac(B)[2]\simeq G_3/G_4$,
    				\item the map $\calC\mapsto \ell_\calC$ induces an isomorphism $H_1(\Gamma)\simeq G_2/G_3$,
    				\item $G_1/G_2\isom(\Z/2)^{m_1}$, where $m_1\in \{0,1\}$, and is generated by $\ell_1$ (if it exists), and
				\item $G_0/G_1\isom(\Z/2)^{m_c}$, where $m_c \in \{0,1\}$, and is generated by $\ell_c,\ell_{c,1}$ (if either exists).
    			\end{enumerate}
                Furthermore, the map $\divv\circ \Norm_{L/K}: L^\times \to \Div(W)$ induces an isomorphism 
			\[
				\divv\circ\Norm_{L/K} : L_\calE/G_1 \isom \langle m_c\divv(c) \rangle \subset \Div(W)\otimes \Z/2\Z, 
			\]
			and
        		\begin{align*}
        			m_1 &= 
        				\begin{cases}
        					1 & \textup{if } \deg(B^0_{w_1}) = \vec{0},\\
        					0 & \textup{otherwise,} 
        				\end{cases}\\
        			m_c &=
        				\begin{cases}
        					1 & \textup{if } \divv(c) \notin 2\Div(W)\textup{ and }\deg(\frakS\cap B^0) \in \left\{ \vec{0},\; \deg(B^0_{w_1}) \right\},\\
        					0 & \textup{otherwise.} 
        				\end{cases}
			\end{align*}
			
    		\end{prop}

            We first prove a similar proposition for $\kk(B)_\calE$.
    		\begin{prop}\label{prop:FiltrationOnB}
    			The group $\kk(B)_\calE$ admits a filtration
    			\[
                    0 \subset \tilde G_4 \subset \tilde G_3 \subset \tilde G_2 \subset \tilde G_1 \subset \tilde G_0 = \kk(B)_\calE\,,
    			\]
    			where 
    			\begin{enumerate}
    				\item $\tilde G_4 = \kk(B)^{\times2}$,
    				\item $\tilde G_3 = \left\{\tilde\ell\in \kk(B)^{\times}:\divv(\tilde\ell)\in 2\Div(B)\right\}$ and $\Jac(B)[2]\stackrel{\sim}{\To}\tilde G_3/\tilde G_4$, where $[D]\mapsto \tilde\ell_D$,
    				\item $\tilde G_2 = \left\{\tilde\ell\in \kk(B)^{\times}:
                        \nu_*\divv(\tilde\ell)\in
                                    2\Div(B^0)\right\}$ 
 and $H_1(\Gamma)\stackrel{\sim}{\To}\tilde G_2/\tilde G_3$, where $\calC\mapsto \tilde\ell_\calC$, 
                \item $\tilde{G}_1=\left\{\tilde\ell\in \kk(B)^{\times}:
                        \psi(\ell)\in \langle  B^0_{w_1}\rangle\subset
                                    \Div(B^0)\otimes\Z/2\right\}$ and $\tilde G_1/\tilde G_2$ is cyclic, generated by $\tilde\ell_1$ (if it exists), and
				\item  $\tilde G_0/\tilde G_1$ is cyclic, generated by either $\tilde\ell_c$ or $\tilde\ell_{c,1}$ (if they exist).
    			\end{enumerate}

    		\end{prop}
            \begin{proof}
                The containments are clear.  It remains to prove the claims about the isomorphisms.  The isomorphism in $(2)$ follows immediately from the definition of $\Jac(B)[2]$.  The isomorphism in $(3)$ follows from~\eqref{eq:homologysquare} and Lemma~\ref{lem:PrincipalDivisors}.  
            
                The map $\psi$ gives a commutative diagram of exact sequences
                \[\xymatrix{
                    0\ar[r]& \tilde{G}_2 \ar[r]\ar@{=}[d]& \tilde{G}_1 \ar[r]^(.35){\psi}\ar@{^{(}->}[d] &\langle  B^0_{w_1}\rangle \ar@{^{(}->}[d] 
                    \\
                    0\ar[r] &\tilde{G}_2 \ar[r] &\tilde{G}_0 \ar[r]^(.35){\psi} &\langle \frakS\cap B^0, B^0_{w_1}\rangle, 
                }\]
                which proves $(4)$ and $(5)$.
            \end{proof}

		\begin{lemma}\label{lem:prBBfl}
			The map $\operatorname{pr}:\kk(B) \to L$ induces an isomorphism, $\kk(B)_{\calE}/\kk(B)^{\times2} \isom L_\calE/L^{\times 2}$ and satisfies
			\[
				\varpi_*\circ\nu_*\circ\divv = \varpi_*\circ\nu_*\circ\divv\circ\operatorname{pr}.
			\]
		\end{lemma}

		\begin{proof}
				Decompose $B$ as a disjoint union, 
		\[
			B = \Bfl \cup \left(\bigcup_{i = 1}^r \bigcup_{j = 1}^{s_i} F_{i,j}\right) = \Bfl \cup \left(\bigcup_{i = 1}^r T_i\right)\,,
		\] 
		where the $F_{i,j}$ are the dimension one irreducible components of $B$ that do not map dominantly to $W$, $T_i = F_{i,1} \cup \dots \cup F_{i,s_i}$, and the indices are arranged so that the $\nu(T_i)\subset B^0$ are pairwise disjoint, connected trees of genus zero curves. 

		Now suppose $\tilde\ell := (\ell^2, f_1, \dots, f_r) \in \kk(B)_{\calE}$, with $f_i \in \kk(T_i)^\times$ at least one of which is not a square. A function $f_i \in \kk(T_i)^\times$ is not a square if and only if $\nu_*\divv(f_i) \subset \Div(B^0)$ has odd valuation on at least two distinct points of $\nu(T_i)$. So, since the $\nu(T_i)$ are disjoint,
		\[
                	\nu_*(\divv(\tilde\ell)) = 2\nu_*(\divv(\ell)) + \nu_*(\divv(f_1)) + \dots + \nu_*(\divv(f_r))
            	\]
		has odd valuation on at least two distinct points of $B^0$ within a single fiber.
		
		On the other hand, by definition of $\kk(B)_\calE$, $\nu_*(\divv(\tilde\ell))\in \langle \frakS\cap B^0, B^0_{w_1}\rangle$.  Since $B^0_{w_1}$ is disjoint from $\nu(T_i)$ for all $i$ and $\frakS$ meets $\nu(T_i)$ in exactly $1$ point, this leads to a contradiction. Thus all of the $f_i$ are squares and this the first statement in the lemma.
	
		The second statement follows from the fact that if $f \in \kk(F_{i,j})^\times$ is a function on an irreducible vertical component of $B$, then $\varpi_*\nu_*\divv(f) = \deg(f)\cdot w = 0$, for some $w \in W$.
		\end{proof}
                        
        \begin{proof}[Proof of Proposition~\ref{prop:Filtration}]
	
	    Set $G_i$ to be the image of $\tilde G_i$ in $L^\times$. Statements (1)--(5)
	 follow immediately from Proposition~\ref{prop:FiltrationOnB} and Lemma~\ref{lem:prBBfl}. Note that since $X^0_{w_1}$ is reduced and irreducible, $B^0_{w_1}$ and $\frakS\cap B^0 + B^0_{w_1}$ are nonzero in $\Div(B^0)\otimes\Z/2$.
Together with the image of $\psi$, which is determined by Proposition~\ref{prop:Existence}, this is enough to yield the formula for $m_1$
	 and deduce that
		\[
			m_c = \begin{cases}
					1 & \textup{if } \frakS\cap B^0\not\in 2\Div(B^0)\textup{ and }\deg(\frakS\cap B^0) \in \left\{ \vec{0},\; \deg(B^0_{w_1}) \right\},\\
					0 & \textup{otherwise.} 
				\end{cases}
		\]
		Since $\varpi_*(\frak{S}\cap B^0) = \divv(c)$ and $\frak{S}$ is a section, we see that
		\[
			\frak{S}\cap B^0 \in 2\Div(B^0) \quad \Leftrightarrow \quad \divv(c) \in 2\Div(W)\,,
		\]
		which gives the formula for $m_c$. The claim regarding $\divv\circ\Norm_{L/K}$ follows easily from Lemma~\ref{lem:prBBfl} and the fact that $B_{w_1}^0$ has even degree.
        \end{proof}
        \begin{remark}
		Note that when $\BOfl \ne B^0$ the inclusion $\{\ell\in L^{\times}:\nu_*(\divv(\ell)) \in 2\Div(\BOfl) \} \subset G_2$ is not an equality, in general. This is a manifestation of the \emph{non}-commutativity of the diagram
    		\[\xymatrix{
    			\Div(B)\ar^{\nu_*}[d]\ar[r] & \Div(\Bfl)\ar^{\nu_*}[d] \\
    			\Div(B^0)\ar[r] & \Div(\BOfl) \\
			}
    		\]
            where the horizontal maps are the natural projections.
        \end{remark}
            

	\section{The presentation of $\Br X[2]$}\label{sec:BrX2presentation}

            \edited{ We define $\frakL_{\calE} \subseteq \frakL$ to be the image of $L_{\calE}$ (defined in the previous section) under the projection map $L^{\times} \to L^{\times}/K^{\times}L^{\times2}$ and we set
 $\frakL_{c,\calE} := \frakL_c\cap\frakL_{\calE}.$}

            \begin{remark}\label{rmk:RationalS}
		\edited{When $S$ is rational (i.e., $W = \PP^1$), $\frakL_{c,\calE} = \frakL_\calE$. See Lemma~\ref{lem:LowerBound}}.
            \end{remark}

     	   The goal of this section is to prove the following theorem.
			
			\begin{thm}\label{thm:BrX2presentation}
				If $S$ is geometrically ruled, then there is an exact sequence
				\[
					\frac{\Pic C}{2\Pic C} \stackrel{x-\alpha}\To
					\frakL_{c,\calE} \stackrel{\gamma}{\To}
					\Br(X\setminus\calE)[2] \To 0\,.
				\]
			\end{thm}
			\begin{remark}
				\label{rem:GaloisAction}
				Suppose that $\pi\colon X\to S$ and the ruling on $S$ are defined over a subfield $k_\circ \subseteq k$ for which $k$ is a separable closure.  Then all abelian groups in Theorem~\ref{thm:BrX2presentation} have an action of $\Gal(k/k_\circ)$ and by Theorem~\ref{thm:EvenHypThm}, the maps in the exact sequence are morphisms of Galois modules.  The same statement holds for the corollaries below.
			\end{remark}

			\begin{cor}\label{cor:BrX2presentation}
				If $S$ is geometrically ruled and $B^0$ has at worst simple singularities, then there is an exact sequence
				\[
					\frac{\Pic C}{2\Pic C} \stackrel{x-\alpha}\To
					\frakL_{c,\calE} \stackrel{\gamma}{\To}
					\Br X[2] \To 0\,.
				\] 
			\end{cor}
			
			\begin{proof}
				Apply Proposition~\ref{prop:-2curves}.
			\end{proof}
			
			\begin{proof}[Proof of Theorem I]
				This follows almost directly from Corollary~\ref{cor:BrX2presentation}.  It remains to extend the map $x - \alpha$ to $\frac{\NS X}{2\NS X}$ and to check compatibility with the Galois action.  The latter follows from Theorem~\ref{thm:EvenHypThm}.  To extend $x - \alpha$, we note that since $k$ is separably closed of characteristic different from $2$, $\Pic^0 X \subseteq 2\Pic X$.  Therefore, the restriction map $\Pic X\to \Pic C$ induces a surjective map $\frac{\NS X}{2\NS X}\to \frac{\Pic C}{2\Pic C}.$  \edited{Since $\Br X$ and $\Pic C$ remain invariant under birational transformations, we may disregard the choice of desingularization.}\Editnote{Comment 15}
			\end{proof}
			
			\begin{cor}\label{cor:DoubleCoverOfP2}
				Let $\pi'\colon X'\to \PP^2$ be a double cover branched over a smooth irreducible curve $B'$ \edited{of degree greater than $2$}\Editnote{Comment 29}.  Then there is a short exact sequence
				\[
					0 \to \frac{\Pic X'}{\langle [H]\rangle + 2\Pic X'} 
					\to \left(\frac{\Pic B'}{K_{B'}}\right)[2] \to 
					\Br X'[2]\to 0\,.
				\]
				{where $[H]\in\Pic X'$ is the pullback of the hyperplane class on $\PP^2$ and $K_{B'}$ is the canonical divisor on $B$'.  In particular, we have} $\dim_{\F_2}\Br X'[2] = 2 + 2(2d - 1)(d - 1) - \rk \Pic X'$, where $2d = \deg(B').$
			\end{cor}
			\begin{remarks}\hfill
                \begin{enumerate}
                    \item[(i)] \edited{If $\pi'\colon X'\to \PP^2$ is a double cover branched over a conic, then $X'$ is a quadric surface and thus $\Br X' = 0$.}\Editnote{Comment 29}
                    \item[(ii)]	If $\pi'\colon X'\to \PP^2$ is branched over a singular curve $B'$ with at worst simple singularities, then Theorem~\ref{thm:BrX2presentation} still applies to give a presentation of $\Br X'[2]$; however, the presentation cannot solely be given in terms of a quotient of $\left({\Pic B'}/{K_{B'}}\right)[2]$.
                \end{enumerate}
			\end{remarks}
			\begin{proof}[Proof of Corollary~\ref{cor:DoubleCoverOfP2}]
				The group $\left(\Pic B'/K_{B'}\right)[2]$ consists of the $2$-torsion classes in $\Jac(B')$ and the theta characteristics, and so it has $\F_2$-dimension $1 + 2g(B') = 1 + 2(2d - 1)(d - 1).$  Therefore, the second claim follows easily from the first.
				
				Fix a point $P\in \PP^2\setminus B'$ such that no line through $P$ is everywhere tangent to $B'$.  Let $S := \textup{Bl}_P \PP^2$, $X := \textup{Bl}_{\pi'^{-1}(P)} X'$, and let $B$ denote the strict transform of $B'$ in $S$, or equivalently, $X$.  Projection away from $P$ gives a {geometric} ruling on $S$, thus we may apply the theorem.  We may choose coordinates such that $\mathfrak{S}$ is the exceptional curve above $P$; then $c = 1$ and the two exceptional curves above $\pi'^{-1}(P)$ correspond to $\infty^+$ and $\infty^-$.  Hence, by Theorems~\ref{thm:EvenHypThm} and~\ref{thm:BrX2presentation}\Editnote{Comment 30: No change needed - see Response.txt}, we have a short exact sequence
				\[
					0\to 
					\frac{\Pic C}{\langle[\infty^+],[\infty^-]\rangle + 2\Pic C}
					\stackrel{x-\alpha}\To
										\frakL_{c,\calE} \stackrel{\gamma}{\To}
										\Br X[2] \To 0\,.
				\]
				
				By adjunction, {$K_{B'} = (2d - 3)[l]|_{B'}$}, and $[l]|_{B'} = [B_{w_1}]$ in $\Pic(B)$, where $[l]\in\Pic \PP^2$ is the class of a line. (Note that $\pi'^*[l] = [H]$.) So the theta characteristics are in bijection with functions $\ell\in L^{\times}$, considered up to squares, such that $\divv(\ell) = {B_{w_1}} + 2D$ for some {$D\in \Div(B)$}.  This, together with our assumption on the point $P$, implies that $\left(\frac{\Pic B'}{K_{B'}}\right)[2]$ is isomorphic to $\frakL_{c,\calE}$.  Furthermore, our assumptions on the point $P$ also implies that we have an isomorphism $\Pic X'/[H]\stackrel{\sim}{\to} \Pic C/\langle[\infty^+],[\infty^-]\rangle$ obtained by composing the pullback map with restriction to the generic fiber.  This completes the proof.
			\end{proof}

			\begin{lemma}\label{lem:L_calE}\Editnote{This lemma is new.}
				Assume that $S$ is geometrically ruled.  Then $\ell\in L^{\times}$ represents an element in $\frakL_\calE$ if and only if there exists a function $a\in K^{\times}$ and integers $n_1,n_c$ and $n_P$ such that
	            \begin{equation}\label{eq:UnramDivisor}
			\nu_*(\divv(a\ell)) \equiv 
			n_{1} B^0_{w_1}	
			+n_c \left(\frakS\cap \BOfl\right)
	                + \sum_{\substack{w \in W\\ S_w\subset B^0}} 
	                    \sum_{P\in \BOfl_w}n_P P
	                \pmod {2\Div(\BOfl)}, 
	            \end{equation}
	            where $\sum_{P\in \BOfl_w}n_P \equiv n_c \bmod 2$ for all $w \in W$ such that $S_w\subset B^0$. 
			\end{lemma}
			
			\begin{proof}
                If $\ell\in L^{\times}$ represents an element in $\frakL_{\calE}$, then there exist $n_1,n_c \in \{0,1\}$ and functions $a\in K^{\times}$ and $f_1, \ldots, f_n \in k(x)^{\times}$ such that
                \[
                    \nu_*\divv(a\ell, f_1, \ldots, f_n) \equiv
                     n_{1}B^0_{w_1} + n_c\left(\frakS\cap B^0\right) \pmod{2\Div(B^0)}
                \]
                Recall that $\kk(B) \isom L\times \prod_{i=1}^n\kk(F_i)$ where $F_1, \ldots, F_n$ are the irreducible components of $B^0$ that do not map dominantly to $W$. Since $S$ is geometrically ruled, the $F_i$ are pairwise disjoint genus zero curves.
		 Since $\frak{S} \cap B^0 = \frak{S}\cap \BOfl + \sum_{i=1}^n\frak{S}\cap F_i$ and $w_1$ was chosen such that $S_{w_1}\not\subset B^0$, this implies that 
                \begin{equation*}
		    \nu_*\divv(a\ell) \equiv 
                    n_{1}B^0_{w_1} + n_c\left(\frakS\cap \BOfl\right) + 
                    \sum_{i=1}^n \sum_{P\in \BOfl\cap F_i}n_PP \pmod{2\Div(\BOfl)},
		\end{equation*}
		where, for all $i$, the integers $n_P$ satisfy
		\begin{equation}\label{eq:divf_i}
                    \nu_*\divv(f_i) \equiv 
                    n_c (\frakS\cap F_i) + \sum_{P\in \BOfl\cap F_i}n_P P \pmod{2\Div(B^0)}.
                \end{equation}
		As $\frak{S}\cap F_i$ has degree $1$ and any principal divisor has degree $0$, it follows that $a\ell$ must satisfy~\eqref{eq:UnramDivisor}.

		To prove the converse, suppose that $a\ell$ satisfies~\eqref{eq:UnramDivisor}. For each $w$ such that $S_w \subset B^0$ there is some $i$ such that $S_w = F_i$ and $\sum_{P\in \BOfl\cap F_i}n_P = \sum_{P \in \BOfl_w}n_P \equiv n_c \bmod 2$. Since the Jacobian of $F_i$ is $2$-divisible, there exists a function $f_i \in \kk(F_i)$ satisfying~\eqref{eq:divf_i}. Then   
		\[
			\nu_*\divv(a\ell,f_1,\dots,f_n) \equiv n_1 B_{w_1}^0 + n_c(\frak{S}\cap B^0)  \pmod{2\Div(B^0)}\,,
		\]
		which implies that $(a\ell,f_1,\dots,f_n) \in \kk(B)_\calE$ and so $\ell$ represents a class in $\frakL_\calE$.
			\end{proof}
			
			\begin{proof}[Proof of Theorem~\ref{thm:BrX2presentation}]\Editnote{This proof has been rewritten.}
				It suffices to show that 
				\[
					\gamma^{-1}\left(\Br (X\setminus\calE)[2]\right) = 
					\frakL_{c,\calE};
				\]
				the statement about the kernel of $\gamma$ and the exactness follows from Theorem~\ref{thm:gammawelldefined}.  
				
				If $\ell\in L^{\times}$ represents an element in $ \frakL_{c,\calE}$, then $\Norm_{L/K}({\ell})\in K^{\times2}\cup cK^{\times2}$ and there exists a function $a\in K^{\times}$ and integers $n_1,n_c$ and $n_P$ such that~\eqref{eq:UnramDivisor} holds.  Thus Theorem~\ref{thm:gammawelldefined} and Proposition~\ref{prop:UnramifiedConditions} imply that $\gamma'(\ell)\in \Br (X\setminus\calE)[2]$.

				{Now let $\ell\in L^{\times}$ be such that $\gamma(\overline{\ell})\in\Br(X\setminus\calE) \subseteq \Br C$. By Theorem~\ref{thm:gammawelldefined}, $\overline{\ell} \in \frakL_c$.  In particular, there exists $n_c \in \{0,1\}$ such that $\varpi_*\divv(\ell) \equiv n_c\divv(c) \bmod{2\Div(W)}.$ By Proposition~\ref{prop:UnramifiedConditions},

		\begin{equation}\label{eq:triplesum}
				\nu_*(\divv(\ell)) \equiv 
			    \sum_{\substack{w \in W\\ S_w\not\subset B^0}}m_w B^0_w
	                    + \sum_{\substack{Q\in \frakS\cap \BOfl \\ }} n_QQ 
	                    + \sum_{\substack{w \in W\\ S_w\subset B^0}} 
	                        \sum_{P\in \BOfl_w}n_P P
	                    \pmod {2\Div(\BOfl)},
	            \end{equation}
	            for some choice of $m_w, n_P, n_Q$ only finitely many of which are nonzero. This does not necessarily specify the $n_P$ and $n_Q$ uniquely modulo $2$, as any point $Q \in \frak{S} \cap \BOfl$ such that $S_{\varpi(Q)} \subset B^0$ is equal to some $P$. We may specify a unique choice modulo $2$ by requiring that $n_Q \equiv n_cv_Q(\frak{S}\cap \BOfl) \mod 2\Div(\BOfl)$, for all $Q \in \frak{S} \cap \BOfl$ such that $S_{\varpi(Q)} \subset B^0$.

			Since the Jacobian of $W$ is $2$-divisible,} there exists an integer $n_{1}$ and a function $a\in K^{\times}$ such that
				\[
					\divv(a) \equiv 
					n_{1}w_1 - 
					\sum_{\substack{w\in W\\S_w\not\subset B^0}}m_ww 
					\pmod{2\Div(W)}.
				\]			
				Thus,
	            \begin{equation}\label{eq:VerticalConditions}
	                \nu_*(\divv(a\ell)) \equiv
			    n_{1} B^0_{w_1}
	                    + \sum_{Q\in \frakS\cap \BOfl} n_QQ 
	                    + \sum_{\substack{w \in W\\ S_w\subset B^0}} 
	                        \sum_{P\in \BOfl_w}n_P P
	                    \pmod {2\Div(\BOfl)}.
	            \end{equation}
            
	            Pushing~\eqref{eq:VerticalConditions} forward to $W$ and using the fact that $\deg(\Bfl/W)$ is even we obtain that
	            \begin{equation}\label{eq:divccongruence}
			n_c\divv(c) \equiv
	                \sum_{Q\in \frakS\cap \BOfl} n_Q\varpi_*Q
	                    + \sum_{\substack{w \in W\\ S_w\subset B^0}} 
	                        \left(\sum_{P\in \BOfl_w}n_P\right)w \pmod{2\Div(W)}
	            \end{equation}
		    Since $\divv(c) \equiv \varpi_*(\frakS\cap B^0)\bmod{2\Div(W)}$, this implies that $n_Q \equiv n_cv_Q(\frak{S}\cap \BOfl) \bmod 2$ for any $Q$ such that $S_{\varpi(Q)} \not\subset B^0$. As we have also arranged this to be true for $Q$ such that $S_{\varpi(Q)} \subset B^0$, it follows that $\sum_{Q \in \frak{S}\cap\BOfl} n_Q Q \equiv n_c(\frak{S}\cap\BOfl) \bmod 2\Div(\BOfl).$
		The congruence in~\eqref{eq:divccongruence} then gives 
		\begin{align*}
			\sum_{\substack{w \in W\\ S_w\subset B^0}} 
	                        \left(\sum_{P\in \BOfl_w}n_P\right)w
				&\equiv n_c\varpi_*\left((\frak{S}\cap B^0) - (\frak{S}\cap\BOfl)\right) &\pmod{2\Div(W)}\\
				&\equiv n_c\sum_{\substack{w \in W\\ S_w \subset B^0}}w &\pmod{2\Div(W)}
		\end{align*}
		It follows that $\sum_{P \in \BOfl_w}n_P \equiv n_c \bmod 2$, for each $w \in W$ such that $S_w\subset B^0$.
	            It now follows from Lemma~\ref{lem:L_calE} that $\overline{\ell}\in \frakL_{c,\calE},$ which completes the proof.
			\end{proof}
				
	\section{The dimension of $\Br X[2]$}\label{sec:sizeBr}
		
		\begin{thm}\label{thm:sizeBr}
			Assume that $S$ is geometrically ruled, that $B^0$ is connected with at worst simple singularities, and that $\Pic^0 X$ \edited{is in the kernel of the \edited{restriction} morphism $\Pic X\to \Pic C$}.\Editnote{Comment 31}  Then
			\edited{\begin{equation*}
                2g(B^0) - 4g(W)  + 3 - \textup{rank}(\NS X)
                \leq \dim_{\F_2}\Br X[2] \leq
                2g(B^0) - 2g(W)  + 4 - \textup{rank}(\NS X)
			\end{equation*}}
			with \edited{the upper bound an }equality \edited{if $S$ is rational.}
		\end{thm}

		\begin{remark}
			Many surfaces satisfy the condition that \edited{$\Pic^0 X$ is in the kernel of the \edited{restriction} morphism $\Pic X\to \Pic C$.}  
              Indeed, K3 surfaces and Enriques surfaces have trivial $\Pic^0$, and the proof of the following corollary shows that the same is true for double covers of rational geometrically ruled surfaces that have connected branch locus.
		\end{remark}
		
		\begin{cor}
			\label{cor:sizeBr}
			Let $\varpi\colon S\to \PP^1$ be a rational geometrically ruled surface with invariant $e \ge 0$, and let $Z$ be a section of $\varpi$ with self-intersection $-e$. Suppose that $B^0$ is a connected curve of type $(2a,2b) \in \Pic(S) \simeq \Z Z\times \Z S_{w_1}$ with at worst simple singularities. Then
			\[
				\dim_{\F_2}\Br X[2] = 4 + 2(2a - 1)(2b - ae - 1) - \rank(\NS X)\,.
			\]
		\end{cor}

		\begin{example}
			If $X$ is a double cover of $\PP^1\times\PP^1$ branched over a $(4,4)$ curve, then $e = 0$, $a=b=2$ and $X$ is a K3 surface. We recover the well known fact that $\dim_{\F_2}\Br X[2] = 22 - \rank(\NS X)$. Note that this argument does not require knowing that $b_2(X) = 22$ or that $\HH^3(X, \Z)_{\textup{tors}} = \{1\}$.
		\end{example}
		
		\begin{proof}[Proof of Corollary~\ref{cor:sizeBr}]
			
			We will prove that $\HH^1(X,\OO_X) = 0$ and so $\Pic^0 X = 0$; hence we may apply the theorem.  \edited{Since} $S$ is rational, 
			 the upper bound in the theorem is sharp. The formula now follows from the adjunction formula, which gives $g(B^0) = (2a - 1)(2b - ae - 1).$
			
			Now we turn to the computation of $\HH^1(X, \OO_X)$, which (given~\cite[Chap. III, Ex. 8.2]{Hartshorne} and the fact that $X^0$ has at worst simple singularities) is isomorphic to $\HH^1(S, \pi_*\OO_X) = \HH^1(S, \OO_S \oplus \OO_S(-a,-b)) = \HH^1(S, \OO_S(-a,-b))$.  
			Then, by Riemann-Roch, we have
			\begin{align*}
				\hh^0(S, \OO_S(-a,-b)) - \hh^1(S, \OO_S(-a,-b)) + \hh^0(S, \OO_S(a - 2, b - e - 2)) \\= \frac12\left(-a(2-a)(-e) - b(2-a) -a(2 + e - b) \right) + 1 = (a - 1)(b - 1) + \frac{ae}2(a-1).
			\end{align*}
			Since $B^0$ is connected, $a > 0$ and so $\hh^0(S, \OO_S(-a,-b)) = 0$.  If $a = 1$, then $\hh^0(S, \OO_S(a-2, b - e - 2)) = 0$ and so $\hh^1(S,\OO_S(-a,-b)) = 0$.  Hence, we may assume that $a \geq 2$; then by~\cite[Chap. V, Lemma 2.4]{Hartshorne}, 
			\[
				\hh^0(S, \OO_S(a - 2,b - e - 2)) = \hh^0(\PP^1, \varpi_*\OO_S(a-2,b - e - 2)).
			\]
			Further, $\varpi_* \OO_S(a-2,b - e - 2)  = (\OO \oplus\OO(-e))^{\otimes(a-2)}\otimes\OO(b - e - 2).$ Since $B^0$ is connected, $g(B^0) \ge 0$ and so $2b - ae - 1\geq0$.  Combined with $a\geq 2$, this shows that $b - e - 2$ is non negative.  Therefore $\hh^0\left(\PP^1, (\OO \oplus\OO(-e))^{\otimes(a-2)}\otimes\OO(b - e - 2)\right) = (a - 1)(b - 1) + \frac{ae}2(a-1),$ which completes the proof.
		\end{proof}
		
        \edited{We first prove a few preliminary results which will be useful in the proof of Theorem~\ref{thm:sizeBr}.}
		\begin{lemma}\label{lem:dimImage} 
			Assume that $\Pic^0 X$ \edited{is in the kernel of the restriction morphism  $\Pic X \to \Pic C$.} Then $\Pic C$ is a finitely generated abelian group and
			\[	
				\dim_{\F_2}\im (x-\alpha) = \textup{rank}\left(\Pic C\right) + h^0(B^\textup{fl}) - 2 +
				\dim_{\F_2}\left(\frac{K^{\times}\cap L^{\times2}}{K^{\times2}}\right)-
				\begin{cases}
					1 & \textup{if }c\in K^{\times2}\\
					0 & \textup{if }c\notin K^{\times2}
				\end{cases}.
			\]
		\end{lemma}
		\begin{proof}
			By assumption $\Pic^0 X$ is contained in the kernel of the restriction morphism $\Pic X \to \Pic C$.  As $\Pic X \to \Pic C$\edited{ is surjective}, we also have a surjection $\NS X \to \Pic C$ and so $\Pic C$ is a finitely generated abelian group.  Thus,
			\begin{align*}
				\dim_{\F_2}\frac{\Pic C}{2\Pic C} & = \rank \Pic C + \dim_{\F_2} \Pic C[2] = \rank \Pic C + \dim_{\F_2} J(K)[2],\\
				\dim_{\F_2}\im(x-\alpha) & = \rank \Pic C + \dim_{\F_2} J(K)[2] - \dim_{\F_2} \ker(x-\alpha).\\
			\end{align*}

			We claim that
			\[
				\dim_{\F_2} \ker(x - \alpha) = 				
				\begin{cases}
				1 & \textup{if }c\in K^{\times2}\\
				0 & \textup{if }c\notin K^{\times2}
				\end{cases}
				+ 
				\begin{cases}
					0 & \textup{if $\Omega$ admits a $G_K$-stable partition  }\\
					& \textup{into two sets of odd cardinality,}\\
					1 & \textup{otherwise.}
				\end{cases}
			\]
			By Theorem~\ref{thm:EvenHypThm}, the kernel is generated by the divisors lying over $\infty\in\PP^1_K$.  Additionally, $\mm := \pi^*\infty$ is in $2\Pic C$ if and only if $f(x)$ has a factor of odd degree~\cite[Lemma 4.2]{CV-BrauerCurves}, 
and	$\mm$ is the only rational divisor above $\infty\in\PP^1_K$ if $c\notin K^{\times2}$.  If $c\in K^{\times2}$, there there is a rational point $\infty^+$ lying over $\infty\in \PP^1_K$, which gives an independent generator of the kernel.

			Therefore, it suffices to prove that
			\begin{equation}\label{eqn:dimJ2}
				\dim_{\F_2} J(K)[2] = 
					h^0(B^\textup{fl}) - 2 +
					\dim_{\F_2}\left(\frac{K^{\times}\cap
					L^{\times2}}{K^{\times2}}\right) +
				\begin{cases}
					0 & \textup{if $\Omega$ admits a $G_K$-stable partition  }\\
					& \textup{into two sets of odd cardinality,}\\
					1 & \textup{otherwise.}
				\end{cases}
			\end{equation}

			Recall that the elements of $J(K)[2]$ correspond to unordered $G_K$-stable partitions of $\Omega$, the branch locus of $C \to \PP^1_K$, into two sets of even cardinality; note that the number of $G_K$ orbits of $\Omega$ is exactly $h^0(B^{\textup{fl}})$.   Unordered $G_K$-stable partitions of $\Omega$ into two sets of even cardinality can arise in essentially two ways: from even degree factors of $f(x)$, and {from quadratic extensions $F/K$ such that $f(x)$ is the norm of a polynomial over $F$.}  The partitions corresponding to even degree factors of $f(x)$ over $K$ generate a subgroup of $J(K)[2]$ of dimension equal to $h^0(B^{\textup{fl}}) - 2$ or $h^0(B^{\textup{fl}}) - 1$, correspondingly as $f(x)$ does or does not have any factors of odd degree. Partitions coming from a factorization over a quadratic extension only occur when the genus of $C$ is odd, and then only if $f(x)$ has no factor of odd degree, in which case they generate a subgroup of $J(K)[2]$ of dimension $\dim_{\F_2}(K^\times\cap L^{\times 2})/K^{\times 2}$. Thus, $\dim_{\F_2}J(K)[2]$ is equal to
				\[
					\begin{cases}
						h^0(B^{\textup{fl}}) - 2	& 
							\text{ if $f(x)$ has a factor of odd degree}\,,\\
						h^0(B^{\textup{fl}}) - 1 	& 
							\text{ if $f(x)$ has no factor of odd degree 
							and $g(C)$ is even}\,,\\
						h^0(B^{\textup{fl}}) - 1 + \dim_{\F_2}\left(\frac{K^\times\cap L^{\times 2}}{K^{\times 2}}\right) 
							&\text{ if $f(x)$ has no factor of odd degree and 
							$g(C)$ is odd}\,.
					\end{cases}
				\]
				When $f(x)$ has a factor of odd degree we clearly have $\dim_{\F_2}\left((K^\times\cap L^{\times 2})/K^{\times 2}\right) = 0$ and that $\Omega$ admits a $G_K$ stable partition into two set of odd cardinality, so~\eqref{eqn:dimJ2} holds.  Now assume that $f(x)$ has no factor of odd degree. When the genus of $C$ is even, $\deg f(x) \equiv 2 \bmod 4$, and so there can be at most one quadratic extension of $K$ contained in $L$. If such an extension exists, then it gives a $G_K$-stable partition of $\Omega$ into two sets of odd cardinality. When the genus of $C$ is odd there cannot be a $G_K$-stable partition of $\Omega$ into two sets of odd cardinality because $\deg f(x) \equiv 0 \bmod 4$.  Thus~\eqref{eqn:dimJ2} holds in all cases, which completes the proof.  
		\end{proof}

		\begin{lemma}\label{lem:dimLcalE}\Editnote{This lemma is new.} 
			The $\F_2$-dimension of $\frak{L}_\calE$ is equal to
			\begin{align*}
				&2g(B) + 2h^0(B) + b_1(\Gamma) - 2g(W) + \dim_{\F_2}((K^{\times}\cap L^{\times2})/K^{\times2})\\
				& - \#\{w\in W : 2|e(b/w) \;\forall b\in \BOfl_w 
				\textup{ or } S_w\subseteq B^0\}	- 2 + m_1 + m_c.
			\end{align*}			
		\end{lemma}

		\begin{proof}
			By Proposition~\ref{prop:Filtration}
			\[
				\dim_{\F_2}\left(\frac{L_\calE}{L^{\times 2}}\right) = 2g(B)  + 2h^0(B) - 2 + b_1(\Gamma) + m_1 +m_c
			\]
			(recall that $g(B) := \sum g(B_i) + 1 - h^0(B)$ where $B_i$ are the connected components of $B$).

			By (the proof of) Lemma~\ref{lem:L_calE}, a function $a \in K^\times$ lies in $L_\calE$ if and only if
			\[
				\divv(a) = n_{w_1}w_1 +
				\sum_{\substack{w\in W\\2|e(b/w) \;\forall b\in B^0_w \\\textup{or } S_w\subseteq B^0}}n_w w+2D,
			\]
			for some $D \in \Div(W)$ and integers $n_w\in \{0,1\}$.  Furthermore, such functions $a$, modulo squares, are in one-to-one correspondence with elements of
			\[
				\Jac(W)[2]\times \{w\in W : 2|e(b/w) \;\forall b\in B^0_w, \textup{ or } S_w\subseteq B^0
				\}.
			\]
			It follows that $\left(L_\calE \cap K^{\times}\right)/L^{\times 2}$, which is equal to $\ker\left(L_\calE/L^{\times 2} \to \frak{L}_\calE\right)$ has dimension,
			\begin{equation}\label{eq:DimOfKernel}
				2g(W) + \#\{w\in W : 2|e(b/w) \;\forall b\in \BOfl_w 
				\textup{ or } S_w\subseteq B^0\}
				- \dim_{\F_2}\left(\frac{K^{\times}\cap L^{\times2}}{K^{\times2}}\right),
			\end{equation}
			which completes the proof.
		\end{proof}

        \begin{prop}\label{prop:dimLEmodIm}\Editnote{This proposition is new.}
            Assume that $B^0$ is connected with at worst simple singularities, $S$ is geometrically ruled, and $\Pic^0 X$ is in the kernel of the restriction morphism $\Pic X \to \Pic C$.  Then
            \[
                \dim_{\F_2}\left(\frakL_\calE/\im (x - \alpha) \right) = 
                2 + 2g(B^0) - 2g(W) - \textup{rank}(\NS X) +  m_1 + m_c   +
                \begin{cases}
                    1 & c\in K^{\times2}\\
                    0 & c\notin K^{\times2}
                \end{cases}\,.
            \]
        \end{prop}
        \begin{remark}
            As we will make a similar argument later under the weaker assumption that $S$ is ruled, but not necessarily geometrically ruled, we will take care to point out when the geometrically ruled hypothesis is used in the proof.
        \end{remark}
        \begin{proof}
            Lemmas~\ref{lem:dimImage} and~\ref{lem:dimLcalE} (together with some rearranging) show that
            \begin{align}
                \dim_{\F_2}\left(\frac{\frakL_\calE}{\im (x - \alpha)} \right) = & 
                \;m_1 + m_c - 2g(W)+
                \begin{cases}
                    1 & c\in K^{\times2}\\
                    0 & c\notin K^{\times2}
                \end{cases}\,
                \\
                & + 2g(B) + h^0(B) + b_1(\Gamma)\label{eq:gB}\\
                & - \left(\textup{rank}\Pic C + \#\left\{w\in W : 
                2|e(b/w)\;\forall b\in\BOfl_w, S_w\not\subset B^0\right\}\right)
                \label{eq:NSX}\\
                & + h^0(B) - h^0(\Bfl) - \#\left\{w\in W : S_w\subset B^0\right\}\label{eq:zero}\,.
            \end{align}

            Let $a_n, d_n, e_6, e_7, e_8\in\Z$ denote the number of $A_n, D_n, E_6, E_7,$ and $E_8$ singularities on $B^0$ respectively (for definitions see \cite[\S II.8]{BHPvdV}).  Recall that the $\delta$-invariant of a singular point $P$ is the difference between the genus of the singular curve and the genus of the curve obtained by resolving the singularity at $P$.  It can be computed using the Milnor number and the number of branches of the singularity~\cite[Thm. 10.5]{Milnor}.  Since
            \[
                \delta(A_n) = \left\lfloor\frac{n+1}{2}\right\rfloor, \;
                \delta(D_n) = \left\lfloor\frac{n + 2}{2}\right\rfloor,\;
                \delta(E_6) = 3, \; \delta(E_7) = \delta(E_8) = 4,
            \]
            the genus of $B$ equals
            \[
                g(B^0) - \sum_n
                \left(a_n\left\lfloor\frac{n + 1}{2}\right\rfloor +
                d_n \left\lfloor\frac{n + 2}{2}\right\rfloor\right)
                - 3e_6 - 4(e_7 + e_8).
            \]
            Furthermore, singularities of type $A_{2k + 1}, D_{2k + 1}$ or $E_7$ each contribute exactly one edge to $\Gamma$, and singularities of type $D_{2k}$ each contribute two edges to $\Gamma$~\cite[Table 1, p.109]{BHPvdV}.  Moreover, $\Gamma$ has $h^0(B)$ vertices and, since $B^0$ is connected, $\Gamma$ has $1$ connected component. Therefore, $b_1(\Gamma) = \sum_{k}(a_{2k + 1} + d_{2k + 1} + 2d_{2k}) + e_7 + 1 - h^0(B)$.  Combining these facts, we have
            \[
                2g(B) + h^0(B) + b_1(\Gamma) = 2g(B^0) - \sum_nn(a_n + d_n) - 6e_6 - 7e_7 - 8e_8 + 1.
            \]
            Since $S$ is geometrically ruled, $\calE$ consists \emph{only} of exceptional curves obtained by blowing up singularities of $B^0$, and so $\#\calE = \sum_nn(a_n + d_n) + 6e_6 + 7e_7 + 8e_8$.  Therefore,~\eqref{eq:gB} simplifies to
            \[
                2g(B^0) - \#\calE + 1.
            \]
            
            Since $\Pic^0 X$ becomes trivial when restricted to $C$, we have a surjective map $\NS X \to \Pic C$.  Therefore, $\rank \NS X = \rank \Pic C + \rank \ker\left(\NS X \to \Pic C\right)$. Further, since $S$ is geometrically ruled, we have
            \[
                \rank \NS X = \rank \Pic C + 
                \#\calE + 1 + \#\{w\in W : 2|e(b/w) \;
                \forall b\in \BOfl_w, \;
                S_w\not\subseteq B^0\},
            \]
            and so~\eqref{eq:NSX} simplifies to
            \[
                - \left(\rank\NS X - \#\calE - 1\right).
            \]
            
            Finally we note that, by definition of $\Bfl$,~\eqref{eq:zero} simplifies to $0$, so
            \[
                \eqref{eq:gB} + \eqref{eq:NSX} + \eqref{eq:zero} = 
                2g(B^0) - \#\calE + 1 - \left(\rank\NS X - \#\calE - 1\right)
                = 2 + 2g(B^0) - \rank\NS X,
            \]
            as desired.
        \end{proof}

        \begin{lemma}\label{lem:LowerBound}\Editnote{This lemma is new.}
			$\dim_{\F_2}(\frak{L}_{c,\calE}) \ge \dim_{\F_2}(\frak{L}_{\calE}) - 2g(W).$
		\end{lemma}
		
		\begin{proof}
			Proposition~\ref{prop:Filtration} gives rise to a commutative and exact diagram,
			\[
				\xymatrix{
					0 \ar[r] 
					&\displaystyle \frac{G_1}{K^\times L^{\times 2}} \ar[d] \ar[r]
					&\frak{L}_\calE \ar[r]\ar[d]^{\Norm_{L/K}}
					&(\Z/2)^{m_c} \ar@{^{(}->}[d]^{1 \mapsto \divv(c)} 
                    {\ar[r]} & 0\\
					0 \ar[r] & \Jac(W)[2] \ar[r]& K^\times/K^{\times 2} \ar[r]& \Div(W)\otimes \Z/2 \ar[r] &0\,.	
				}
			\]
            Let $I$ denote the image of the vertical map on the left. Then by the snake lemma we have
            \begin{equation*}
             \dim_{\F_2}\frakL_{c,\calE} 
             \geq \dim_{\F_2}\frakL_\calE\cap\frakL_1
             =  \dim_{\F_2}\left(G_1/K^{\times}L^{\times2}\right)\cap\frakL_1
             = \dim_{\F_2}\frakL_\calE - m_c - \dim_{\F_2} I.
            \end{equation*}
           The lemma now follows unless $m_c + \dim_{\F_2} I > 2g(W)$, i.e., unless $m_c=1$ and $I = \Jac(W)[2]$. In this case either $\ell_c$ or $\ell_{c,1}$ exists.  Assume that $\ell_c$ exists and let $a = c\Norm_{L/K}(\ell_c)$. Since $\divv(c) \equiv \divv(\Norm_{L/K}(\ell_c))\bmod{2\Div(W)}$, $\overline{a}$ lies in the image of $\Jac(W)[2]\to K^{\times}/K^{\times2}$.  Since $I = \Jac(W)[2]$, there exists some $\overline{\ell}\in G_1/K^{\times}L^{\times2}$ with $\Norm_{L/K}(\overline{\ell}) = \overline{a}$.  Therefore, $\overline{\ell_c}/\overline{\ell}\in \frakL_{c,\calE}\setminus \frakL_\calE\cap\frakL_1$, and so we conclude that
            \begin{equation*}
             \dim_{\F_2}\frakL_{c,\calE} 
             = 1 + \dim_{\F_2}\frakL_\calE\cap\frakL_1
             = 1 + \dim_{\F_2}\frakL_\calE - m_c - \dim_{\F_2} I = \dim_{\F_2}\frakL_\calE - 2g(W).
            \end{equation*}
            If $\ell_c$ does not exist, then the same argument may instead be applied to $\ell_{c,1}$.
		\end{proof}

		\begin{proof}[Proof of Theorem~\ref{thm:sizeBr}]\Editnote{This proof has been rewritten.}
            We first use Proposition~\ref{prop:dimLEmodIm} to obtain bounds on $\dim_{\F_2}\left(\frakL_\calE/\im (x - \alpha) \right)$.
            Since $B^0$ is connected and $B_{w_1}$ has even degree, Proposition~\ref{prop:Filtration} shows that $m_1 = 1$.  If $c\in K^{\times2}$ then $\divv(c) \in 2\Div(W)$, and the converse holds if $g(W) = 0$.  Further, since $S$ is geometrically ruled and $B^0$ is connected, $\deg(\frakS\cap B^0) = \vec{0}$.  Hence, by Proposition~\ref{prop:Filtration} we have
            \[
                0 \leq m_c + \begin{cases}
                    1 & c\in K^{\times2}\\
                    0 & c\notin K^{\times2}
                \end{cases} \leq 1,
            \]
            with the upper bound an equality if $g(W) = 0$.
            Combining this with Proposition~\ref{prop:dimLEmodIm}, we obtain
            \begin{equation}\label{eq:UpperLE}
                0 \leq \dim_{\F_2}\left(\frakL_\calE/\im (x - \alpha) \right) - 
                \left(2g(B^0) - 2g(W)  + 3 - \textup{rank}(\NS X)\right) \leq 1\,,
            \end{equation}
            with the upper bound an equality if $g(W) = 0$.
            
            Now we will relate $\dim_{\F_2}\Br X[2]$ to $\dim_{\F_2}\left(\frakL_\calE/\im (x - \alpha) \right)$.
            Corollary~\ref{cor:BrX2presentation} readily yields,
            \[
                \dim_{\F_2}\Br X[2]
                = \dim_{\F_2}\frakL_{c,\calE} - \dim_{\F_2}\im (x - \alpha)
                \le \dim_{\F_2}\frakL_{\calE} - \dim_{\F_2}\im (x - \alpha)\,,
            \]
            which combined with~\eqref{eq:UpperLE} gives the desired upper bound.
            Finally, Lemma~\ref{lem:LowerBound} together with~\eqref{eq:UpperLE} yields the desired lower bound.
		\end{proof}

	\section{The Brauer group of an Enriques surface}\label{sec:Enriques}
		An \defi{Enriques surface} is a smooth projective minimal surface ${E}$ with nontrivial $2$-torsion canonical divisor and with irregularity ${h^1(\OO_E) = 0}$.  Equivalently, an Enriques surface is a quotient of a K$3$ surface by a fixed point free involution.  The Brauer group of any Enriques surface (over a separably closed field) is isomorphic to $\Z/2\Z$~\cite[p. 3223]{HS-Enriques}. In this section we compute a central simple algebra representing the nontrivial class.
		
		We shall see below that every Enriques surface is birational to a double cover of a ruled surface{ whose branch locus has at worst simple singularities.}  This implies that every Enriques surface is birational to a double cover of a \emph{geometrically} ruled surface; however, the branch locus of this double cover may have worse singularities.  We will find it more convenient to adapt the methods of the previous sections to ruled surfaces which fail to be geometrically ruled.
		
		\subsection{Horikawa's representation of Enriques surfaces}
		\label{subsec:Horikawa}
			 Let $E$ be an Enriques surface and let $\tilde E$ be its K3 double cover. Horikawa's representation of Enriques surfaces~\cite[Chap VIII, Props. 18.1, 18.2]{BHPvdV} shows that $\tilde{E}$ is the minimal resolution of a double cover of a quadric surface $\tilde{S}^0\subseteq\PP^3$ branched over a reduced curve $\tilde{B}^0$, which has at worst simple singularities, and which is obtained by intersecting a quartic hypersurface with $\tilde{S}^0$.  Furthermore, the covering involution $\sigma\colon \tilde{E} \to \tilde{E}$ for the quotient $\tilde{E}\to E$ descends to the involution 
			\[
				\tau\colon\PP^3\to \PP^3\,, \quad (z_0:z_1:z_2:z_3)\mapsto(z_0:-z_1:-z_2:z_3)\,.
			\]
			Therefore, $\tilde{B}^0$ is invariant and fixed point free under the action of $\tau$.  Under the Horikawa representation, we may take $\tilde{S}^0$ to be the quadric cone $V(z_0z_3 - z_1^2)$ if $E$ is \defi{special} and $\tilde{S}^0 = V(z_0z_3 - z_1z_2)$ otherwise. (An Enriques surface is \defi{special} if it is endowed with the structure of an elliptic pencil together with a $(-2)$-curve which is a $2$-section, and \defi{nonspecial} otherwise.)  If $\tilde{S}^0$ is non-singular, then the {morphism}
			\[
				\tilde{S}^0 \to \PP^1_{\tilde{t}}\,, \quad \vec{z}\mapsto (z_1:z_0) = (z_3:z_2)
			\]
			shows that $\tilde{S} := \tilde{S}^0$ is a rational geometrically ruled surface.  If $\tilde{S}^0$ is the quadric cone, then the rational map
			\[
				\tilde{S}^0\dasharrow \PP^1_{\tilde{t}}\,, \quad \vec{z}\mapsto (z_1:z_0) = (z_3:z_1)\,.
			\]
			shows that the blow up $\tilde{S} := \textup{Bl}_{(0:0:1:0)}(\tilde{S}^0)$ is a rational geometrically ruled surface. In either case, $\tilde{E}$ is birational to the double cover $\tilde{X}^0$ of $\tilde{S}$ branched over $\tilde{B}^0$ (where we abuse notation using $\tilde{B}^0$ to denote its strict transform in $\tilde{S}$). 		

			We may embed $\tilde{S}/\tau$ in $\PP^4$ as the vanishing of $V(w_0w_3 - w_4^2, w_1w_2 - w_4^2)$.  Under this embedding, the morphism $\tilde{S} \to \tilde{S}/\tau$ is given by $(z_0:z_1:z_2:z_3)\mapsto (z_0^2:z_1^2:z_2^2:z_3^2:z_0z_3)$.  The ruling on $\tilde{S}$ induces a map $\varpi\colon \tilde{S}/\tau\to\PP^1_t, \mathbf{w}\mapsto (w_1:w_0) = (w_3 : w_2)$ giving $\tilde{S}/\tau$ the structure of a ruled surface. {Note that the coordinates $\tilde t$ and $t$ on the two copies of $\PP^1$ are related by $\tilde{t} = \sqrt{t}$.}
			
			Let $S := \textup{Bl}_{\textup{Sing}(\tilde{S}/\tau)}(\tilde{S}/\tau)$. Then $E$ is birational to the double cover $X^0$ of $S$ branched over the exceptional divisors on $S$ and the strict transform $B^0$ of $\tilde{B}^0/\tau$. We let $X$ and $\tilde X$ be the desingularizations of $X^0$ and $\tilde X^0$ obtained by canonical resolutions. Observe that, in agreement with the convention set in \S\ref{subsec:SurfacesNotation}, $B^0$ contains all connected components of the branch locus that map dominantly to $\PP^1$. We may thus avail ourselves of the notation and results established in the previous sections for both $X/S$ and $\tilde{X}/\tilde{S}$, using tildes to denote objects corresponding to $\tilde X$.
		
			\begin{remark}
			As a caution, we note that $S$ is not geometrically ruled; the fibers above $0$ and $\infty$ consist of a chain of three arithmetic genus $0$ curves with the center curve appearing with multiplicity $2$. 
			 \end{remark}
			
		\subsection{The Brauer group}
			\begin{theorem}\label{thm:BrauerEnriques}
				For every $[D]\in\Jac(\Bfl)[2]$ and cycle $\calC$ on the dual graph of $B^0$, the algebra $\gamma'(\ell_{\calC}\ell_D)$ lies in $\Br X$.  Moreover, these elements generate $\Br X$.
			\end{theorem}
			
			\begin{proof}
				Fix a cycle $\calC$ and a divisor $D$ on $\Bfl$ whose divisor class is $2$-torsion.  We claim that $\partial_{F}(\gamma'(\ell_{\calC}\ell_D))\in \kk(F)^{\times2}$ for all reduced and irreducible curves $F \subseteq X$, and thus that $\gamma(\ell_{\calC}\ell_D)\in \Br X$.  If $F$ is a horizontal curve then this follows from Theorem~\ref{thm:gammawelldefined}, since $\overline{\ell_\calC\ell_D} \in \frakL_c$.  Now assume that $F$ is a vertical curve.  If $F$ does not map dominantly to the reduced part of a component of $S_0$ or $S_{\infty}$, then the preimage of $F$ in $\tilde{X}$ consists of exactly two curves $F_1, F_2$, each of which are isomorphic to $F$.  We will show that $\partial_{F_i}(\textup{Res}(\gamma'(\ell_{\calC}\ell_D)))\in \kk(F_i)^{\times2}$, and thus conclude that  $\partial_{F}(\gamma'(\ell_{\calC}\ell_D))\in \kk(F)^{\times2}$.  Since $\kk(\tilde{B}) = \kk(\tilde{B}/\tau)\otimes_{k(t)}k(\sqrt{t})$, we have
				\[
					\textup{Res}(\gamma'(\ell_{\calC}\ell_D)) = 
					\Cor_{\kk(\tilde{X}_{\tilde{B}})/\kk(\tilde{X})}\left(
					\textup{Res}\left((x - \alpha, \ell_{\calC}\ell_D)_2\right)
					\right).
				\]
				We may choose our coordinates $x$ and $\alpha$ on the Enriques surface so that $\Res(x) = \tilde{x}/\tilde{t}$ and $\Res(\alpha) = \tilde{\alpha}/\tilde{t}$, where $\tilde{x}$ and $\tilde{\alpha}$ are the functions on $\tilde{E}$ and $\tilde{B}$.  Therefore
				\[
					\Cor_{\kk(\tilde{E}_{\tilde{B}})/\kk(\tilde{E})}\left(
					\textup{Res}\left((x - \alpha, \ell_{\calC}\ell_D)_2\right)
					\right) = 
					\tilde\gamma'(\Res(\ell_{\calC}\ell_D)) + (\tilde{t}, \Cor(\ell_{\calC}\ell_D))_2.
				\]
				Since $\ell_{\calC}\ell_D\in \frakL_1$, the algebra $(\tilde{t}, \Cor(\ell_{\calC}\ell_D))_2$ is trivial in the Brauer group.  Furthermore, by construction $\Res(\ell_{\calC}\ell_D) = \ell_{\tilde{\calC}}\ell_{\tilde{D}}$ for some cycle $\tilde\calC$ on the dual graph of $\tilde{B}$ and some two-torsion divisor $\tilde{D}$ on $\tilde{B}$.  Hence, by applying Corollary~\ref{cor:BrX2presentation} to $\tilde{X}/\tilde{S}$ it follows that $\partial_{F_i}(\gamma'(\ell_{\calC}\ell_D))\in \kk(F_i)^{\times2}$

				Now assume that $F$ maps dominantly to the reduced part of a component of $S_0$. If $F$ maps dominantly to an exceptional divisor of $S$, then $x-\alpha$ has trivial valuation and reduces to a constant on all curves $F'$ that lie above $F$ in the desingularization of $X\times_{\PP^1} B$. Therefore, $\partial_{F}(\gamma'(\ell_{\calC}\ell_D))\in \kk(F)^{\times2}$.  Now consider the case when $F$ maps dominantly to the reduced part of $(\tilde{S}/\tau)_0$. If the singular locus of $B^0$ is supported away from the fibers of $0$, then, after adjusting $D$ by a principal divisor, we may assume that $v_{b}(\ell_{\calC}\ell_{D}) = 0$ for all $b\in B\cap(\varpi^{-1}(0))$. Then $\partial_{F}(\gamma'(\ell_{\calC}\ell_D))$ is a constant and so it is clear that $\partial_{F}(\gamma'(\ell_{\calC}\ell_D))\in \kk(F)^{\times2}$.  If there is a singularity of $B^0$ lying over $0$, then $F$ must be rational. Since there are no nontrivial \'etale covers of a rational curve, the preimage of $F$ in $\tilde{X}$ consists of exactly two curves $F_1, F_2$, each of which are isomorphic to $F$, and we may apply the same argument used above.  The case where $F$ maps dominantly to a component of $(S_\infty)_{\textup{red}}$ follows similarly.
								
				We have shown that the subspace of $\frakL_c$ generated by the $\ell_{\calC}$ and the $\ell_D$ maps into $\Br X$; now we will use a cardinality argument to show that the image of this subspace is all of $\Br X$. Arguing as in the proof of \edited{Lemma~\ref{lem:dimLcalE}} we see that the functions $\ell_{\calC}$ and $\ell_{D}$ generate a subspace of $\frakL' \subseteq \frakL$ of $\F_2$-dimension 
				\begin{align*}
					\dim_{\F_2}\frakL' &=
					2g(B)+2h^0(B) + b_1(\Gamma) -1
					+ \dim_{\F_2}\left(\frac{K^{\times}\cap 
					L^{\times2}}{K^{\times2}}\right)\\				
					&\quad\quad-\#\{w\in\PP^1 : 2|e(b/w) \;\forall b\in \BOfl_w
					\textup{ or }S_w\subseteq B^0\}\,.
					\intertext{Using that $h^0(B) = h^0(\Bfl) + \#\{w\in \PP^1 : S_w\subseteq B^0\}$, we get}
					\dim_{\F_2}\frakL' &=
					2g(B)+h^0(B) + b_1(\Gamma) -1 + h^0(\Bfl)
					+ \dim_{\F_2}\left(\frac{K^{\times}\cap 
					L^{\times2}}{K^{\times2}}\right)\\					
					&\quad\quad-\#\{w\in\PP^1 : 2|e(b/w) \;\forall b\in \BOfl_w
					\textup{ and }S_w\not\subseteq B^0\}\,.
					\intertext{Since $\tilde{B}^0$ is connected and has at most simple singularities, the same is true for $B^0$.  Therefore, the same argument as in the proof of \edited{Proposition~\ref{prop:dimLEmodIm}} shows that $2g(B) + h^0(B) + b_1(\Gamma) -1 = 2g(B^0) - \#\calE'$, where $\calE'$ is the set of exceptional curves obtained by blowing up the singularities of $B^0$. (We have $\#(\calE \setminus \calE') = 4$ corresponding to the extra irreducible components of $S_0$ and $S_w$ -- see Remark~\ref{rmk:exceptional}.) Therefore}
					\dim_{\F_2}\frakL' &=
					2g( B^0) + h^0(\Bfl)
					+ \dim_{\F_2}\left(\frac{K^{\times}\cap 
					L^{\times2}}{K^{\times2}}\right)
					\\					
					&\quad\quad-\#\calE'-\#\{w\in\PP^1 : 2|e(b/w) \;\forall b\in \BOfl_w
					\textup{ and }S_w\not\subseteq B^0\}\,.
					\intertext{Now noting that $\Pic^0 X = 0$ and so the rank of $\NS X$ is 
						\[
							\#\calE' + 4 + 1 + 
							\#\{w\in \PP^1\setminus\{0,\infty\} : 2|e(b/w) \;\forall b\in \BOfl_w 
							\textup{ and }S_w\not\subseteq B^0\}
							+\rank(\Pic C)\,,
						\]
						we conclude that 
					}
					\dim\frakL'&= 2g(B^0) + h^0(\Bfl)
					+ \dim_{\F_2}\left(\frac{K^{\times}\cap 
					L^{\times2}}{K^{\times2}}\right) +5 + \rank(\Pic C)
					\\					
					&\quad\quad-\rank \NS X - \Delta_0 - \Delta_\infty,
					\intertext{where $\Delta_w$ equals $1$ if $e(b/w)$ is even for all $b\in \Bfl_w$ and $(S_w)_{\textup{red}}\not\subseteq B^0$ and $0$ otherwise.  Since $X_0$ and $X_\infty$ are not reduced and ${\tilde{B}^0}$ did not contain the fixed points of $\tau$, both $\Delta_0$ and $\Delta_\infty$ are $1$.  In addition, $\rank(\NS X) = \rank(\NS E) + 4 = 10 + 4$, and since $B^0$ is the quotient of a genus $9$ curve by a fixed-point free involution, $g(B^0) = 5$. Hence, rearranging and using Lemma~\ref{lem:dimImage} we obtain,}
					\dim\frakL'&= h^0(\Bfl) - 1
					+ \dim_{\F_2}\left(\frac{K^{\times}\cap 
					L^{\times2}}{K^{\times2}}\right)
					+ \rank(\Pic C)\\
					&\ge \dim_{\F_2}\im(x-\alpha) + 1\,.
				\end{align*}
				Therefore, some element of $\frakL'$ has nontrivial image in $\Br \kk(X)$ and so must be equal to the unique nontrivial element of $\Br X$.
			\end{proof}
		\begin{remark}
			If $B^0$ is smooth, it is possible to prove Theorem~\ref{thm:BrauerEnriques} without using that the N\'eron-Severi group of an Enriques surface has rank $10$ and that the Brauer group of an Enriques surface is $\Z/2\Z$.  One can instead prove that all other functions in $\frakL_{c,\calE}$ are ramified along some vertical divisor.  However, this proof is more complicated as it requires a detailed study of the desingularization (or at least the normalization) of the fiber product $X\times_{\PP^1} \Bfl$.
		\end{remark}

	\section{An Enriques surface failing weak approximation}\label{sec:BMobs}
		We demonstrate the previous results are amenable to explicit computation by exhibiting an Enriques surface with a transcendental Brauer-Manin obstruction to weak approximation.  We note that it was already known that Enriques surfaces need not satisfy weak approximation due to work of Harari and Skorobogatov~\cite{HS-Enriques}, who constructed an Enriques surface whose \'etale-Brauer set was strictly smaller than its Brauer set.
		
		\subsection{Construction of the Enriques surface}

		Let $\tilde{S}$ denote the quadric surface $V(z_0z_3 - z_1z_2) \subseteq \PP^3$ and let $\tilde{B^0}\subseteq \tilde{S}$ be the (reducible) quartic curve given by the vanishing of
		\[
			F(\mathbf{z}) := 
			\left(z_3^2  - 3z_2^2 - 3z_1^2 - 2z_0^2 + 3z_3z_0 \right)^2 - 
			\left(z_3z_1 - 2z_3z_2 + 4z_0z_2 + z_0z_1\right)^2. 
		\]
		We let $\tilde{E}$ denote the minimal {resolution }
		of the double cover 
		\[
			\tilde{E^0} := V(y^2 - F(\mathbf{z}), 
			z_0z_3 - z_1z_2)\subseteq\PP(1,1,1,1,2);
		\]
		note that $\tilde{E}$ is a K3 surface defined over $\Q$.  There is a fixed point free involution $\sigma^0\colon \tilde{E^0} \to \tilde{E^0}, \quad (z_0:z_1:z_2:z_3:y) \mapsto (z_0:-z_1:-z_2:z_3:-y)$ that can be lifted to a fixed point free involution $\sigma$ on $\tilde{E}$.

		Let $E$ denote the Enriques surface $\tilde{E}/\sigma$. We denote the base change of $E$ to an algebraic closure of $\Q$ by $\Ebar$.
		
		\begin{prop}
			\label{prop:Example}
			There exists a number field $\F$ such that 
			\[
				E(\A_{\F})^{\Br E_{\F}} \subsetneq E(\A_{\F})^{\Br_1 E_{\F}}.
			\]
		\end{prop}
		\begin{remark}
			The surface $X$ from the previous section is a blow-up of the Enriques surface $E$.  Since the Brauer group is a birational invariant of smooth projective surfaces, $\Br \Ebar = \Br \Xbar$.  We will use this equality throughout this section. 
		\end{remark}

		\subsection{Coordinates on the generic fiber}
			The map $\PP^4 \to \PP^1, \textbf{w}\mapsto (w_1:w_0)$, induces a morphism $\varpi\colon \tilde{S}/\tau \to \PP^1$ whose generic fiber is isomorphic to $\PP^1_{\Q(t)}$.  The generic fiber of $E\to \tilde{S}/\tau\to \PP^1$ has a model of the form $v^2 = c f(x)$, where
			\[
				f(x) = x^4 + \frac{10t - 2}{t - 10 + 9/t}x^3 
				+ \frac{-7t^2 + 19t - 4}{t - 10 + 9/t}x^2
				+ \frac{-20t^2 - 20t}{t - 10 + 9/t}x
				+ \frac{9t^3 + 11t^2 + 4t}{t - 10 + 9/t}
			\]  
			and $c = t - 10 + 9/t$.  
			The isomorphism between this model and $E$ identifies $t$ with $w_3/w_2 = w_1/w_0 = w_4^2/(w_0w_2)$ and identifies $x$ with $w_4/w_0 = w_3/w_4 = tw_2/w_4$.
		\subsection{Geometry of the branch curve}
			Let $E^0 := \tilde E^0/\sigma^0$; the morphism $E^0\to \tilde S/\tau$ is branched over the four singular points of $\tilde S/\tau$ and the irreducible curve $B^0 := \tilde{B}^0/\tau$. The singular locus of $B^0$ is a degree $5$, $0$-dimensional reduced subscheme.  It consists of one $\Q$-point, $P_0 = (1:1:1:1:1)$, which corresponds to the singular point of an irreducible component of $B^0$, and a degree $4$ point which is irreducible over $\Q$.  The curve {$B^0$} is embedded in $\PP^4$ as the complete intersection of $3$ quadrics so it has arithmetic genus $5$.  One can check that each singularity is an ordinary double point, thus {$B^0$} is geometrically rational.  A naive point search quickly finds smooth $\Q$-points, so {$B^0$ is birational to $\PP^1_\Q$}  and $L \isom \Q(s)$.  We fix the following isomorphism between $L$ and $\Q(s):$
			\[
				\frac{w_1}{w_0} = 
				\left(\frac{2s^2 - 16s + 41}{-s^2 - s + 29}\right)^2, \quad 
				\frac{w_2}{w_0} = 
				\left(\frac{2s^2 - 7s + 32}{-s^2 + 8s + 20}\right)^2, \quad
			\]
			$w_4/w_0 = \sqrt{w_1w_2/w_0^2}$, and $w_3/w_0 = w_1w_2/w_0^2.$  Using the above expression for $f(x)$, we see that this isomorphism sends
			\[
				\alpha\mapsto
				\frac{4s^4 - 46s^3 + 258s^2 - 799s + 1312}{s^4 - 7s^3 - 57s^2 + 212s + 580}.
			\]

		\subsection{Representing a transcendental Brauer class}
		\label{subsec:nontrivialBrauer}
			
			Let $\F'$ denote the residue field of the singular degree $4$ $\Q$-subscheme, and let $P_1$ denote a $\F'$-point of the singular locus different from $P_0$.  (In fact, $\F'$ is an $S_4$ extension, so there is a unique such $P_1$.) The field $\F$ in Proposition~\ref{prop:Example} will be an extension of $\F'$. We write $x_1 := w_4/w_0(P_1)$ and $t_1 := w_1/w_0(P_1)$.  Let $\ell$ be the monic quadratic separable $\F'$-polynomial in $s$ whose zeros lie above $P_1$.  By Theorem~\ref{thm:BrauerEnriques}, $\gamma'(\ell)$ is contained in $\Br \Ebar$. We claim that it represents the nontrivial element in $\Br \Ebar$.  
			
			Using linear algebra and elimination ideals, we find equations for curves on $\tilde{S}$ which pass through an even number of the $\Qbar$ points that are Galois conjugate to $P_1$, and meet ${B^0}$ with even multiplicity at every point of intersection.  There are finitely many such curves $Z$, and one checks that for every set of singular points of ${B^0}$ containing an even number of the $\Gal(\Qbar/\Q)$-conjugates of $P_1$, there is such a reduced and irreducible curve $Z$ such that $\pi^{-1}(Z)$ is reducible.  By computing intersection numbers, one sees that these curves, the rulings on $\tilde{S}$, and the exceptional curves generate $\Pic \Ebar$.  Therefore, the above curves and the horizontal ruling on $S$ generates $\Pic \Cbar_{\Qbar(t)}$.  By the construction of these curves, any function in the image of $x - \alpha$ that has odd valuation at both points in $\nu^{-1}(P_1)$ will have odd valuation at both points in $\nu^{-1}(P_i)$ for an odd number of points $P_i$ that are Galois conjugate to $P_1$.  Therefore $\overline\ell\not\in \im(x-\alpha)$, and so, by Theorem~\ref{thm:BrauerEnriques}, $\gamma'(\ell)$ is nontrivial in $\Br \Ebar.$
			
			Now we will compute a number field $\F$ such that $\gamma'(\ell) \in \im(\Br E_{\F} \to \Br \Ebar)$.  Since $\ell$ is defined over $\F'$, we have $\gamma'(\ell)\in \Br\kk(E_{\F'})$.  However, a direct computation shows that $\gamma'(\ell)$ is ramified at all of the $(-2)$-curves.  There are two linearly independent quadratic extensions of $\F'$ over which the residues at the exceptional curve above $P_0$ and $P_1$ respectively become trivial.  Then there is a degree $4$ extension over the composite of these quadratic extensions over which the residue at the exceptional curve above the degree three singular point becomes trivial.\footnote{A \texttt{Magma}~\cite{MAGMA} script verifying these claims and all other computational claims in this section can be found with the arXiv distribution of this article.}  We let $\F$ denote this degree $16$ extension over $\F'$.
			
			We claim that $\gamma'(\ell) \in \im(\Br E_{\F}\to \Br \Ebar)$, or, more precisely, that the algebra
			\[
				A :=  \gamma'(\ell) + (t - 1,\ell(10))_2 + (t - 9, \ell(-2))_2
			\]
			lies in $\Br E_{\F}$. The algebras $(t - 1,\ell(10))_2$ and $(t - 9, \ell(-2))_2$ are algebraic, so it is clear that the second claim implies the first.  To prove that $A\in \Br E_{\F}$, we must show that $\partial_{D}(A)\in\kk(D)/\kk(D)^{\times2}$ for all prime divisors $D$ on $X$.  From the definition of $A$ and $\F$ and the proofs of Theorem~\ref{thm:gammawelldefined} and Proposition~\ref{prop:UnramifiedConditions}, this is certainly true except possibly for the fibers above $t = 0,\infty$ and for $\mathfrak{S}_\infty$.  One can directly check that $\Norm_{\F'(s)/\F'(t)}(\ell)\in \F'(t)^{\times2}$, so by Theorem~\ref{thm:gammawelldefined}, $\gamma'(\ell)$, and therefore $A$, is unramified at $\mathfrak{S}_{\infty}.$  To compute the residue at the fibers above $t=0,\infty$, we note that $(\tilde{S}\times_{\PP^1_{\sqrt{t}}}\tilde{B})/\tau$ is a smooth birational model of $\tilde{S}/\tau \times_{\PP^1_t}B$.  Then a direct computation using~\eqref{eq:resatF} shows that $\tilde\gamma'(\ell)$ and $\gamma'(\ell)$ are unramified at the fibers above $t = 0,\infty$.  Thus $A\in\Br E_{\F}\setminus\Br_1 E_{\F}$.

		\subsection{Computing the obstruction}

			In this section we will construct a $\F$-adelic point on $E$ that is not orthogonal to $A$.  First note that $\tilde{E}$, and therefore $E$, has $\Q$-rational points.  Indeed, $(1:0:0:0:\pm 2)\in \tilde{E}^0(\Q)$, and since $\tilde{E}\to \tilde{E}^0$ is an isomorphism when $y\neq 0$,  $Q := \psi(1:0:0:0:\pm 2)\in {E}(\Q)$ (here $\psi\colon \tilde{E}\to E$ denotes the quotient map).  Hence, if we find a place $v$ of $\F$ and a point $Q_v\in E(\F_v)$ such that $\inv_v A(Q_v) \neq \inv_v A(Q)$, then the adelic point that is equal to $Q$ for all places $w \neq v$ and equal to $Q_v$ at $v$ is orthogonal to $A$.

			We will take $v$ to be a place lying over $2$.  We note that $2$ splits completely in $\F'$, and of these four places lying over $2$, there is a unique place and a unique extension $v$ of that place to $\F$ such that $\F_v = \Q_2(i).$  Since $A\in\Br \kk(E_{\F'})$ and $A$ is unramified on an open set of $E_{\F'}$ containing $Q$, $A(Q)\in \Br \F'$ and, consequently, $A(Q)\otimes_{\F}\F_v \in\im(\Br \Q_2\to \Br \Q_2(i))$.  Thus $\inv_v A(Q) = 0$.  Let
			\[
				Q_v := \psi(1 : -6 : 1 + i : -6 - i : 2\sqrt{4255 - 4160i}).
			\]
			The point $Q_v$ lies over $t = 36$ and ${B}_t$ consists of $2$ $\F_v$-points $R_1$ and $R_2$ and one quadratic point $R$; they have $s$ values
			\[
				\frac{19}4, \quad \frac{-7}2, \quad\textup{ and }\quad
				\frac{-11 + 3\sqrt{109}}4.
			\]
			
			Since these points are unramified in ${B}\to \PP^1$ and are away from the support of $\ell$ and $\alpha$, the cocycle description of $\gamma'(\ell)$ in~\cite[Lem. 3.5]{CV-BrauerCurves}
			 shows that
			\begin{align*}
				A(Q_v) &= (\ell(-2), 25)_2 + (\ell(-10), 35)_2 + \gamma'(\ell)(Q_v) 
				\\
				& = 
				(\ell(10), 35)_2 +
				\Cor_{\Q_2(i,\sqrt{109})/\Q_2(i)}\left((x(Q_v) - \alpha(R), \ell(R))_2\right)\\
				&\phantom{=}\quad+ \sum_{j=1}^2 (x(Q_v) - \alpha(R_j), \ell(R_j))_2.
			\end{align*}
			A computation shows that $(x(Q_v) - \alpha(R_j), \ell(R_j))_2$ is trivial in $\Br \Q_2(i)$ for all $i$, that $(\ell(10), 35)_2$ is trivial in $\Br \Q_2(i)$, and that $\Cor_{\Q_2(i,\sqrt{109})/\Q_2(i)}\left((x(Q_v) - \alpha(R), \ell(R))_2\right)$ is nontrivial in $\Br \Q_2(i)$.  Therefore $\inv_v A(Q_v) = 1/2.$
			
		\subsection{Determining the algebraic Brauer classes}
			It remains to prove that this failure of weak approximation is not accounted for by algebraic Brauer classes.  In~\ref{subsec:nontrivialBrauer}, we outlined how to obtain generators for $\Pic \Ebar$.  Given these generators, a standard, although involved, computation shows that $\HH^1(G_\F, \Num \Ebar) = 0$, and so $\HH^1(G_\F, \langle K_E\rangle)\twoheadrightarrow \HH^1(G_\F, \Pic \Ebar)$.  Then, the isomorphism from the Hochschild-Serre spectral sequence and~\cite[Thm. 6.2.1]{Skorobogatov-torsors} shows that $\psi(\tilde{E}(\A_\F))\subseteq E(\A_\F)^{\Br_1}$.  Since the adelic point considered in the previous section is contained in $\psi(\tilde{E}(\A_\F))$, this shows that the adelic point above lies in $E(\A_\F)^{\Br_1}$, as desired.

		\section*{Acknowledgements}
			The second author would like to thank Dan Abramovich, Asher Auel, Jean-Louis Colliot-Th\'el\`ene, and Bjorn Poonen for helpful conversations.


\end{document}